\documentclass[a4paper]{article}

\usepackage[style=numeric, sorting=none, maxnames=99, url=false, doi=false]{biblatex}
\usepackage[dvipsnames]{xcolor}
\usepackage[english]{babel}
\usepackage[colorlinks=true, allcolors=blue]{hyperref}
\usepackage{algorithm}
\usepackage{algpseudocode}
\usepackage{subcaption}
\usepackage{tikz}
\usetikzlibrary{math}
\usetikzlibrary{arrows}
\usepackage{authblk}
\usepackage[letterpaper,top=2cm,bottom=2cm,left=3cm,right=3cm,marginparwidth=1.75cm]{geometry}

\usepackage{amsmath,amsfonts,bm,amsthm}
\usepackage{commath}
\usepackage{amssymb}
\usepackage{enumerate}
\newcommand{\etal}{{\it et al.}}
\newtheorem{theorem}{\textbf{Theorem}}[section]

\newtheorem{lemma}[theorem]{\textbf{Lemma}}
\newtheorem{example}[theorem]{\textbf{Example}}

\newtheorem{definition}[theorem]{\textbf{\textbf{Definition}}}
\newtheorem{remark}[theorem]{\textbf{Remark}}

\def\eqref#1{Equation (\ref{#1})}
\def\thmref#1{Theorem \ref{#1}}
\def\defref#1{Definition \ref{#1}}
\def\exref#1{Example \ref{#1}}

\def\figref#1{Figure~\ref{#1}}

\def\gE{{\mathcal{E}}}
\def\gF{{\mathcal{F}}}
\def\gG{{\mathcal{G}}}
\def\gL{{\mathcal{L}}}
\def\gN{{\mathcal{N}}}
\def\gP{{\mathcal{P}}}
\def\gV{{\mathcal{V}}}
\def\sP{{\mathbb{P}}}
\def\sR{{\mathbb{R}}}

\def\vw{{\bm{w}}}

\DeclareMathOperator*{\lipv}{\text{Lip}_1(\gV)}
\newcommand{\lipxv}[1]{\text{Lip}_{1,#1}(\gV)}

\newcommand{\vol}{\text{Vol}}
\newcommand{\diam}{\text{diam}}

\title{Ricci Flow on Weighted Digraphs with Balancing Factor}
\author[1]{Shuliang Bai}
\author[2]{Rui Li}
\author[2]{Shuang Liu} 
\author[1]{Xin Lai\footnote{Corresponding author: laixin@bimsa.cn}}

\affil[1]{{\small Beijing Key Laboratory of Topological Statistics and Applications for Complex Systems, Beijing Institute of Mathematical Sciences and Applications, Beijing, 101408, China.} }
\affil[2]{{\small School of Mathematics, Remin University of China, Beijing, 100872, China}}

\begin{document}
\maketitle
\begin{abstract}
Ricci curvature and Ricci flow have proven to be powerful tools for analyzing the geometry of discrete structures, particularly on undirected graphs, where they have been applied to tasks ranging from community detection to graph representation learning. 
However, their development on directed graphs remains limited, with Ricci flow being especially underexplored. 
In this work, we introduce a rigorous formulation of Ricci flow on directed weighted graphs, which evolves edge weights while preserving distances, and establish both the existence and uniqueness of its solutions.
To capture the essence of asymmetry in directed networks and to enhance the capability of modeling more flexible structures, we incorporate a node-wise balancing factor that regulates between outflow and inflow. 
Building on the continuous Ricci flow evolution framework, we propose a discrete Ricci flow algorithm that is applicable to numerical computing. Numerical studies on various  directed graph examples demonstrate the capacity of the proposed flow to reveal structural asymmetry and dynamic evolutions.
\end{abstract}


\section{Introduction}
Graphs are pervasive in representing data structures in areas such as Internet of Things (IoT)~\cite{Li_2025_topology, Li2023GraphpoweredLM}, biology~\cite{Yao_2024_Molecular}, social network~\cite{Wang_2023_Minority}, supply chain analysis~\cite{Xiang_2025_graph}, etc.
Directed graphs in particular capture causal/preferential/directional relationships between entities. 
In recent years, discrete Ricci curvature has emerged as a powerful tool to analyze network structure. 

Ricci curvature and Ricci flow are fundamental notions in Riemannian geometry, playing a significant role in understanding manifold geometry. 
Introduced by Hamilton~\cite{Hamilton_1982_three}, the Ricci flow evolves a Riemannian metric $g(t)$ using Ricci curvature $\mathrm{Ric}(g)$ according to
$$
\frac{\partial g}{\partial t} = -2 \,\mathrm{Ric}(g),
$$
progressively smoothing metric irregularities and driving the manifold toward canonical geometries. 
This technique led to significant breakthroughs, most notably Perelman's resolution of the Poincar\`e conjecture and Thurston's geometrization conjecture \cite{Perelman_2002_entropy, Perelman_2003_Ricci}. 
Following these landmark results, Ricci flow has become a central tool in the study geometric evolution.

Motivated by these developments, researchers have sought discrete analogues of Ricci curvature and Ricci flow on graphs. 
Several notions of discrete Ricci curvature have been developed for graphs. 
Among the most widely studied are Ollivier's Ricci curvature~\cite{Ollivier_2009_ricci} on metric spaces using optimal transportation, Lin--Lu--Yau (LLY) curvature~\cite{Lin_2011_Ricci} on graphs using a limit form of Ollivier's Ricci curvature that is better adapted to graphs, and Forman Ricci curvature~\cite{Forman_2003_Bochner, Sreejith_2016_Forman} on graphs using combinatorial approaches. 
Each of them has solid theoretical foundations and computational properties.
Ollivier and LLY curvatures maintain a strong connection to the continuous theory but require computing Wasserstein distances, which can be expensive for large networks. Forman curvature offers simplicity and scalability but lacks some geometric richness.
Building on these definitions of Ricci curvature, discrete Ricci flows have been proposed for theoretical studies and applications in network analysis, including network robustness, community detection, bottleneck identification~\cite{Tan_2024_analyzing, Lai2023Deeper, alon2021on}. 
However, most existing work has focused on \emph{undirected graphs}, where symmetric interactions simplify the underlying structure. 
The study of Ricci flow on \emph{directed graphs}, which inherently exhibit asymmetry and directional dependence, remains largely unexplored. This gap motivates this work to adapt the Ricci flow models to directed and weighted networks, capable of capturing their anisotropic geometry.

Directed weighted graphs are ubiquitous in modeling real-world systems. 
They exhibit inherent \emph{asymmetry} that significantly influences network structure and dynamics. 
Recent surveys on machine learning for directed graphs highlight that geometric approaches remain far less developed than their counterparts for undirected graphs \cite{Sun_2024_towards}. 
This scarcity hinders the practical effectiveness for inherently directed data, such as road networks for traffic forecasting and control \cite{Feng_2024_traffic, Pung_2022_Aroad}, gene regulatory networks for directional biochemical reaction control \cite{Wei_2024_inference}, and social network for information spreading analysis \cite{Wang_2024_information}.  
Simply adapting directed data to algorithms which are designed for undirected data is critical because many complex systems cannot be accurately represented without accounting for directional dependencies and imbalanced interactions. 
Incorporating directionality into Ricci curvature and Ricci flow algorithms is therefore essential for understanding anisotropic geometry, improving network robustness, and guiding structural optimization.

Furthermore, since the LLY's Ricci curvature on graphs is based on the optimal transportation between probability distributions, to capture directional heterogeneity, we introduce a node-wise \emph{balancing factor} $\beta$, which regulates the ratio of outflow to inflow at each node. 
This mechanism provides a richer representation of flow asymmetry and facilitates the directed Ricci flow adapting to both directionality and local structural imbalance.

Another notable limitation of most existing Ricci flow formulations on graphs is that they do not distinguish between \emph{edge weight}  and the \emph{distance metric} of the network. 
For instance, in the traditional Ollivier Ricci flow~\cite{bai_Ollivier_2024}, edge weights are updated by formulas such as
$$
\frac{d}{dt}w_{e}(t) = -\kappa_{e}(t)\,w_{e}(t), \quad
\text{or} \quad
w_{xy}^{(i+1)} = \bigl(1 - \epsilon\,\kappa_{xy}^{(i)}\bigr)\,d^{i}(x,y),
$$
and other variants~\cite{Ni_2019_Community, Lai2022Normalized, Ma_2025_modified}, implying that distances evolve alongside weights, coupling the network’s distance-related geometric structure with its weight-related dynamic properties. 
Although such coupling has proven effective for tasks like community detection, it becomes problematic in scenarios where physical distances, such as road lengths, are inherently fixed, while weights (such as road width or capacity) evolve along Ricci flow process. 
These formulations conflate the dual role of edge weights, i.e., the connection strength and metric distance, obscuring the distinct functions of them.

In this work, we introduce a Ricci flow that decouples edge weight from the metric distance, allowing the weights to evolve while maintaining the static distance metric.
This design better reflects operational dynamics in real-world systems, such as traffic networks \cite{Gao_2019_measuring}, communication systems where latency remains constant while bandwidth evolves, and transport or logistics networks where geography is immutable but flow conditions change. 
In a nutshell, this separation will enable more applicable scenarios for further studies.

The main contributions of this work are threefold. 
Firstly, we extend the concept of Ricci flow to \emph{weighted directed graphs} and propose a new Ricci flow formulation that \emph{preserves the underlying distance metric while evolving edge weights}, decoupling distance-related geometry from weight-related dynamics.
Secondly, we introduce a node-wise \emph{balancing factor} $\beta$, which flexibly regulates the bias between outflow- and inflow-based transport distributions, enabling the model to capture local directional properties and flow asymmetry. 
Lastly, we prove the existence and uniqueness of the solution to the proposed Ricci flow framework for weighted directed graphs, providing a rigorous theoretical foundation for Ricci flow evolution in weighted directed networks. 

The remainder of this paper is structured as follows. Section \ref{sec:preliminaries} reviews definitions and properties of directed graphs. Section \ref{sec:directed_curvature} introduces the directed LLY's Ricci curvature with a balancing factor. Section \ref{sec:directed_flow} formulates the directed Ricci flow and proves the existence and uniqueness of its solution. Section \ref{sec:discrete_flow} gives the discrete time version  of our Ricci flow formulation, making it applicable to computational network analysis. Section \ref{sec:exam_app} presents illustrative examples and simulations. Finally, Section \ref{sec:conclusion} concludes and outlines future directions.

\section{Preliminaries}\label{sec:preliminaries}

A directed graph (digraph) $\gG = (\gV, \gE)$ consists of a finite vertex set $\gV$ and a set of directed edges $\gE\subseteq \gV\times \gV$, and $n=|\gV|$, $m=|\gE|$.
Different from the edges in undirected graphs, a directed edge is an ordered pair of nodes, denoted as $(x,y)$ or $x\to y$, where $x$ is called the \emph{tail} and $y$ is called the \emph{head} of the edge. 
The digraphs in this work are simple, that is, $\gG$ has no loops ($x\to x$) nor multiple edges ($x\rightrightarrows y$). 
In this work, we set the edge weight and distance separately, which is more flexible and in line with more application scenarios. 
Set $w:\gE\to \sR^+$ as the edge weights, and $d:\gV\times\gV\to \sR^+$ as the distance function.
In this part, we give the definitions and fundamental settings.

\begin{definition}[Directed path]
For any $x,y\in \gV$, a directed path from $x$ to $y$ is a sequence of distinct vertices, denoted as $p:=v_0\to v_1\to\dots \to v_k$ such that $v_0=x, v_k=y$, and for each $i = 0, 1, \dots, k-1$, the directed edge $(v_i \to v_{i+1})\in \gE$. 
The length of the path is defined as 
$$l(p):=\sum_{i=0}^{k-1}d_{v_iv_{i+1}}.$$
The set of all directed paths connecting $x$ and $y$ is denoted as $\sP(x, y)$.
\end{definition}

\begin{definition}[Directed distance]
    For any $x,y\in \gV$, the distance between them on $\gG$ is the shortest length of all directed path connecting them, i.e.,
    $$d(x, y) := \inf_{p \in \sP(x, y)} l(p),$$
    If there is no directed path from $ x $ to $ y $, set $ d(x, y) = \infty $.
\end{definition}
Note that the directed distance $ d(x, y) $ satisfies the following properties:

\begin{itemize}
    \item \textbf{Non-negativity:} $d(x, y) \geq 0$ for all $ x, y \in \gV $.
    \item \textbf{Triangle inequality:} $d(x, z) \leq d(x, y) + d(y, z)$ for all $ x, y, z \in \gV $.
    \item \textbf{Asymmetry:} $d(x, y) $ not necessarily equals $ d(y, x)$ for $x, y \in \gV$.
\end{itemize}
The diameter of a directed graph is the length of the longest directed path, that is, 
$$\diam(\gG)=\sup_{x,y\in \gV}d(x,y).$$

A directed graph is said to be strongly connected if for any pair of vertices $x,y\in \gV$, there exists a directed path from $x$ to $y$ and vice versa, thus $\diam(\gG)<\infty$.
By ignoring the directions of edges, a directed graph becomes a related undirected graph. Weak connectivity of a directed graph refers to that the corresponding undirected graph is connected. 
In this work, we assume all directed graphs are strongly connected.

\begin{definition}[Neighborhood]
    If $x\to y$ is a directed edge in $\gG$, then $x$ is an in-neighbor of $y$, and $y$ is an out-neighbor of $x$. Thus the in-neighborhood of any vertex $x\in V$ is the vertices $x$ is pointed to, denoted as
    $$\gN^{in}(x):=\{z\in V\vert z\to x \in E\}.$$
    Similarly, the out-neighborhood indicates the set of all out-neighbors, denoted as
    $$\gN^{out}(x):=\{z\in V\vert x\to z \in E\}.$$
    The neighborhood of $x$ is $\gN(x):=\gN^{in}(x)\cup\gN^{out}(x)$.
\end{definition}
Let $\deg^{in}_x=|\gN^{in}(x)|$, $\deg^{out}_x=|\gN^{out}(x)|$ and $\deg_x=|\gN(x)|$ be the in-degree, out-degree and degree of $x$ respectively. If $\deg_x=0$, then $x\in\gV$ is called an isolated vertex.

\begin{definition}[$c$-Lipschitz function on directed graphs]
    Let $\gG=(\gV, \gE, w, d)$ be a weighted directed graph. 
    $d:\gV\times\gV\to \sR^+$ is a directed distance function. 
    For any real-valued function $f:\gV\to \sR$, $f$ is a $c$-Lipschitz function if and only if there exists a constant $c > 0$ such that for any pair of vertices $x,y\in\gV$, the following inequalities holds:
    $$f(y)-f(x)\leq cd(x,y).$$
\end{definition}
Specifically, if $f$ is a 1-Lipschitz function, that implies $f(y)-f(x)\leq d(x,y)$ for all $x,y\in\gV$. Let $\lipv := \{f:\gV\to \sR \vert f \text{ is 1-Lipschitz function}\}$. Note that $\lipv$ is not compact because we can drift the functions in it by a constant. Chose a base point $x_0\in \gV$, and let $\lipxv{x_0}=\{f\in \lipv\vert f(x_0)=0\}$. Then $\lipxv{x_0}$ is compact for 
\begin{enumerate}
    \item[(1)] \textbf{Uniform Boundedness}: $\vert f(x)\vert \leq \vert f(x) - f(x_0)\vert + \vert f(x_0) \vert \leq \max\{d(x_0, x), d(x,x_0)\} \leq \diam(\gG)$ for $\forall x\in\gV$. Thus $||f||_\infty \leq \diam(\gG)$.
    \item[(2)] \textbf{Closeness}: Let $\{f_n\}\subset \lipxv{x_0}$ and $f_n\to f$ as $n\to\infty$. Then $f(x_0)=\lim_{n\to\infty}f_n(x_0)=0$, and $\forall x,y\in\gV$, $f(y)-f(x)\leq \limsup_{n\to\infty}[f_n(y)-f_n(x)]\leq \limsup_{n\to\infty}d(x,y)=d(x,y)$. Thus $f\in \lipxv{x_0}$.
\end{enumerate}

\section{Ricci Curvature on Directed Graphs}\label{sec:directed_curvature}
\subsection{Probability Distribution and Laplacian}
The probability distribution on directed graph $\gG=(\gV,\gE)$ is a function $\mu:\gV\to [0,1]$ such that $\sum_{x\in \gV}\mu(x)=1$. For any $x\in \gV$ in a graph, it can interact with others via edge connections. The interaction can be illustrated by a well designed probability distribution. But for directed graph, the interaction is more complicated because of the in-/out-neighboring relationships. Next, we give our definition on directed graph based on the transition perspective.
\begin{definition}[Probability transition kernel on directed graphs]
Given a weighted directed graph $\gG=(\gV,\gE, w)$, for any $x\in \gV$, $\deg_x\neq 0$. If $\gN^{out}_x\neq\emptyset$, the out probability transition kernel $P:\gV\to \sR$ is defined as:
\begin{equation*}P(x,z) = 
\left\{
\begin{aligned}
&\frac{w_{xz}}{\sum_{u:x\to u}w_{xu}}, & x\to z;\\
&0, & otherwise.
\end{aligned}
\right.
\end{equation*}
Similarly, if $\gN^{in}_x\neq\emptyset$, the in probability transition kernel is defined as 
\begin{equation*}P^\prime(x,z) = 
\left\{
\begin{aligned}
&\frac{w_{zx}}{\sum_{u:u\to x}w_{ux}}, & z\to x;\\
&0, & otherwise.
\end{aligned}
\right.
\end{equation*}
Let $\beta: \gV\to [0, 1]$ be a flexible balance factor. The 
balancing probability transition kernel
is defined as:
\begin{equation}\label{eq:balance_ptk}
    \gP(x,z):=
    \left\{
    \begin{aligned}
        &P(x,z),        && \text{if } \gN^{in}_x = \emptyset \text{ and }\gN^{out}_x \neq \emptyset;  \\
        &P^\prime(x,z), && \text{if } \gN^{in}_x \neq \emptyset \text{ and }\gN^{out}_x = \emptyset;  \\
        &\beta(x) P(x,z)+ (1-\beta(x)) P^\prime(x,z), && otherwise.\\
    \end{aligned}
    \right.
\end{equation}
Thus, $\textnormal{supp }P(x,\cdot)=\gN^{out}(x)$, $\textnormal{supp }P^\prime(x,\cdot)=\gN^{in}(x)$ and $\textnormal{supp }\gP^\prime(x,\cdot)=\gN(x)$.
\end{definition}
Thus we always have $\sum_{z\in\gV}\gP(x,z)=1$ for $\forall x\in \gV$. Since $x$ is not an isolated vertex, $\gN^{in}_x$ and $\gN^{out}_x$ are not $\emptyset$ at the same time.



\begin{remark}
The factor $\beta(x)$ serves as a balancing factor between the contributions of the out-neighborhood and in-neighborhood of a vertex $x$ in the definition of $\gP(x,z)$. Specifically:
\begin{itemize}
    \item $\beta(x)=\frac{1}{2}$ assigns equal attention to both out- and in-neighbors of $x$, which is discussed in \textnormal{\cite{ozawa_geometric_2020}};
    \item $\beta(x)=\frac{\deg^{\mathrm{out}}_x}{\deg_x}$ ensures that each incident edge of $x$, regardless of its direction, receives equal attention, and it summarizes the three cases in \eqref{eq:balance_ptk}, because $\gN^{in}_x = \emptyset$ ($\gN^{out}_x = \emptyset$) implies $\deg^{in}_x=0$($\deg^{out}_x=0$);
    \item $\beta(x)<\frac{1}{2}$ emphasizes the \emph{in-neighborhood} of $x$, making in-edges more efficient and likely to \emph{decrease} in-edge weight under the Ricci flow;
    \item $\beta(x)>\frac{1}{2}$ emphasizes the \emph{out-neighborhood} of $x$, making out-edges more efficient and likely to \emph{decrease} out-edge weight under the Ricci flow.
\end{itemize}
Intuitively, the choice of $\beta(x)$ allows one to control the relative importance of outgoing and incoming edges in curvature computation, providing flexibility to adapt to different network structures. See~\exref{ex:nonsymmetric_triangle} for an example.
\end{remark}

\begin{remark}
We only talk about strongly connected graphs in this work, thus for any $x\in\gV$, $\gN^{in}_x$ and $\gN^{out}_x$ are both nonempty. So the balancing probability transition kernel writes 
\begin{equation}\label{eq:gP_strong_conn}
    \gP(x,z)=\beta(x) P(x,z)+ (1-\beta(x)) P^\prime(x,z).
\end{equation}
\end{remark}

\begin{definition}[Probability distribution on directed graphs]\label{def:prob_def}
    Let $\gG=(\gV,\gE, w)$ be a weighted directed graph, and $\gP$ be the average probability transition kernel. For any $x\in \gV$, the probability distributions is defined as:
    \begin{equation}\label{eq:mu}
        \mu_x^\alpha(z) = 
            \left\{
                \begin{aligned}
                &\alpha, & z= x;\\
                &(1-\alpha)\gP(x,z), & otherwise,
                \end{aligned}
            \right.
    \end{equation}
    where $\alpha\in [0,1]$.
\end{definition}

\begin{definition}[Laplace operator]
    Let $\gG=(\gV,\gE, w)$ be a weighted directed graph, and $f:V\to\sR$ is a function defined on the graph. For any vertex $x\in\gV$, the Laplace operator on directed graphs is defined as 
    $$\gL f(x) := \sum_{y\sim x}\gP(x,y)(f(x)-f(y))=f(x)-\sum_{y\sim x}\gP(x,y)f(y),$$
    and 
    $$\Delta f(x):=-\gL f(x),$$
    where $y\sim x$ indicates $y$ is an in-/out-neighbor of $x$, i.e., $y\in\gN(x)$.
\end{definition}
\noindent Thus $\gL$ is a functional operator which can be written as $\gL f = (I-\gP)f$. Here, $(If)(x)=f(x)$ for $x\in \gV$ is an identity operator, thus continuous. $\gP$ is a continuous operator, because $\gG$ is finite graph, then for any $f,g$, it holds that 
$$||\gP f -\gP g||=\max_{x\in \gV} \left|\sum_{y\in\gV}\gP(x,y)[f(y)-g(y)]\right|\leq \max_{x\in\gV}\sum_{y\in\gV}|\gP(x,y)|\ ||f-g||_\infty\leq ||f-g||_\infty,$$ 
where $||f||_{\infty}=\max_{x\in\gV} |f(x)|$. Hence $\gL$ is continuous.
Similar as Lemma 3.1 in \cite{ozawa_geometric_2020}, the relation between probability distribution and Laplacian holds:
\begin{equation}\label{eq:sum_f_mu}
    \begin{aligned}
        \sum_{y\in \gV}f(y)\mu_x^\alpha(y) 
        & = \alpha f(x) + (1-\alpha)\sum_{y\sim x}\gP(x,y)f(y) \\
        & = \alpha f(x) + (1-\alpha)(\Delta f(x) + f(x)) \\
        & =f(x) + (1-\alpha)\Delta f(x).
    \end{aligned}
\end{equation}
This equation is useful when analyzing the dual of Wasserstein transportation distance, which we will discuss in the next section.

\subsection{Optimal Transport and Ricci Curvature}
The optimal Transport problem is related to minimize the total transportation cost between two probability distributions~\cite{villani_topics_2003}. 
To formalize this optimal problem, it is straightforward to define a transportation plan and a cost function. Let $\mu$ and $\nu$ are two probability distributions. $\pi:\gV\times \gV\to [0,+\infty)$ is a \textbf{transportation plan} from $\mu$ to $\nu$ if it satisfies:
$$\sum_{y\in\gV}\pi(x,y)=\mu(x), \quad\text{and}\quad\sum_{x\in\gV}\pi(x,y)=\nu(y).$$
We also say that $\pi$ is a \textbf{coupling} from $\mu$ to $\nu$. 
The set of all couplings from $\mu$ to $\nu$ is denoted as $\Pi(\mu, \nu)$, which is non-empty since there is a trivial transportation plan $\pi=\mu\otimes\nu \in \Pi(\mu, \nu)$ ('$\otimes$' indicates the tensor product). The \textbf{cost function} is denoted as $c(x,y)$ which indicates the costs to transport one unit from $x$ to $y$. It is natural that $c$ is non-negative and measurable.
The minimizing problem is formulated as:
\begin{equation}
    \min_{\pi\in\Pi(\mu, \nu)} \sum_{(x,y)\in \gV\times \gV}c(x,y)\pi(x,y).
\end{equation}
The cost function $c(x,y)$ is usually related to the distance between $x$ and $y$.
\begin{definition}[Wasserstein distance]
    Let $\gG=(\gV,\gE, w, d)$ is a weighted directed graph where $d$ is a directed distance function on $\gG$. $\mu$ and $\nu$ are two probability distributions on $\gG$. The \textbf{$\bm{L_1}$-Wasserstein distance} is the infimum of transportation cost, that is,
    \begin{equation}
    W_1(\mu, \nu):=\inf_{\pi\in\Pi(\mu, \nu)} \sum_{(x,y)\in \gV\times \gV}\pi(x,y)d(x,y).
    \end{equation}
The optimal transportation plan is denoted as $\pi^\ast$.
\end{definition}

Since the directed distance $d$ is generally asymmetric, the Wasserstein distance $W_1$ is also asymmetric, i.e., $W_1(\mu,\nu)\neq W_1(\nu,\mu)$ in general. 
However, $W_1$ still satisfies the non-negativity and the triangle inequality.
Note that for strongly connected finite graphs, by the duality theorem of a linear optimization problem\cite{villani_topics_2003}, the Wasserstein distance has a strong duality representation written as:
\begin{equation}\label{eq:W_dual}
    W_1(\mu, \nu)=\sup_{f\in 1\text{-}Lip} \sum_{x\in \gV}f(x)(\nu(x)-\mu(x)),
\end{equation}
where $1\text{-}Lip$ is the set of all 1-Lipschitz function on $\gV$. Here is a useful convexity property of Wasserstein distance obtained from \cite{villani_topics_2003} (Section 7.4). 
\begin{lemma}[Convexity of Wasserstein distance]
    For any four probability distributions $\mu_1, \mu_2, \nu_1, \nu_2$ and $\lambda \in [0,1]$, 
    \begin{equation}
        W_1(\lambda\mu_1+(1-\lambda)\nu_1, \lambda\mu_2+(1-\lambda)\nu_2) \leq \lambda W_1(\mu_1, \mu_2) + (1-\lambda)W_1(\nu_1, \nu_2).
    \end{equation}
\end{lemma}
\noindent The proof can be easily obtained by the dual formula~\eqref{eq:W_dual}. 
Similar to the definition of Ricci curvature on undirected graphs~\cite{Lin_2011_Ricci}, we give the definition of Ricci curvature on directed graphs.

\begin{definition}[Ricci curvature]\label{def:LLY_def}
Let $\gG=(\gV,\gE, w, d)$ be a directed weighted graph, and $x,y\in\gV$ be any two vertices. 
    \begin{enumerate}[(1)]
        \item \textbf{$\bm{\alpha}$-Ricci curvature}: Given $\alpha\in[0,1]$,
        $$\kappa_\alpha(x,y) = 1 - \frac{W_1(\mu_x^\alpha, \mu_y^\alpha)}{d(x,y)}.$$
        It is trival that $\kappa_1(x,y)=0$ for all $x,y\in \gV$.
        \item \textbf{Lin-Lu-Yau (LLY) Ricci curvature}:
        $$\kappa(x,y) = \lim_{\alpha \to 1} \frac{\kappa_\alpha(x,y)}{1 - \alpha}$$
    \end{enumerate}
\end{definition}

Note that the weights $w$ and distance $d$ are defined independently. 
We prove this is well-defined and can be expressed in terms of a discrete gradient of the Laplacian.
\begin{lemma}[Concavity]\label{lemma:concavity}
For any $\alpha_1, \alpha_2\in [0,1]$, and $\lambda\in [0,1]$, let $\bar{\alpha}=\lambda\alpha_1+(1-\lambda)\alpha_2$. Then for $\forall x,y\in \gV$, 
$$\kappa_{\bar{\alpha}}(x,y)\geq \lambda\kappa_{\alpha_1}(x,y)+(1-\lambda)\kappa_{\alpha_2}(x,y).$$
\end{lemma}
By the convexity of Wasserstein distance, the concavity of $\kappa_\alpha$ holds as in~\cite{Lin_2011_Ricci, ozawa_geometric_2020}.
Then we get that $\frac{\kappa_\alpha}{1-\alpha}$ is an increasing function. 
According to the concavity of $\kappa_\alpha$ for $\alpha\in [0,1]$,  $\kappa''_\alpha\leq 0$.
Let $f(\alpha):=\frac{\kappa_\alpha}{1-\alpha}$, then $f^\prime(\alpha)=\frac{\kappa^\prime_{\alpha}(1-\alpha)+\kappa_{\alpha}}{(1-\alpha)^2}$. Let $g(\alpha):=\kappa^\prime_{\alpha}(1-\alpha)+\kappa_{\alpha}$, then $g^\prime(\alpha)=\kappa''_{\alpha}(1-\alpha)\leq 0$. Thus $g(\alpha)$ is decreasing. We know that $g(1)=\kappa_1=0$, then $g(\alpha)\geq0$. Thus $f^\prime(\alpha)\geq 0$.

Next we show that $\frac{\kappa_\alpha}{1-\alpha}$ is upper bounded. In 2011, Lin-Lu-Yau~\cite{Lin_2011_Ricci} proved that $\frac{\kappa_\alpha}{1-\alpha}\leq \frac{2}{d(x,y)}$ on undirected graphs. Ozawa \etal~\cite{ozawa_geometric_2020} generalized it to directed graphs in 2020. In this work, we generalized the definition of probability transition kernel using a balancing factor $\beta$, but the upper bounded property still holds. 
\begin{lemma}[Upper bounded]\label{lemma:upper_bounded}
    For any $\alpha\in [0,1]$, $\frac{\kappa_\alpha(x,y)}{1-\alpha}\leq \frac{2\textnormal{\diam}(\gG)}{d(x,y)}$.
\end{lemma}
\begin{proof}
    Let $\delta_x(z):=1$ when $z=x$ and $\delta_x(z):=0$ otherwise.
    Notice that 
    $$W_1(\delta_x, \mu_x^{\alpha}) + W_1(\mu_x^{\alpha}, \mu_y^{\alpha}) + W_1(\mu_y^{\alpha},\delta_y) \geq W_1(\delta_x, \delta_y).$$
    Thus 
    \begin{align*}
        W_1(\mu_x^{\alpha}, \mu_y^{\alpha})
        &\geq W_1(\delta_x, \delta_y) - W_1(\delta_x, \mu_x^{\alpha}) - W_1(\mu_y^{\alpha},\delta_y)\\
        & = d(x,y) - W_1(\delta_x, \mu_x^{\alpha}) - W_1(\mu_y^{\alpha},\delta_y).
    \end{align*}
    Where 
    \begin{align*}
        W_1(\delta_x, \mu_x^{\alpha}) 
        & = \sum_{z\in\gN(x)}(1-\alpha)\gP(x,z)d(x,z) \\
        & \leq (1-\alpha)\diam(\gG)\sum_{z\in\gN(x)}\gP(x,z) \\
        & \leq (1-\alpha)\diam(\gG).
    \end{align*}
    Similarly, $W_1(\mu_y^{\alpha},\delta_y)\leq (1-\alpha)\diam(\gG)$. 
    Then we have 
    $$W_1(\mu_x^{\alpha}, \mu_y^{\alpha}) \geq d(x,y)-2(1-\alpha)\diam(\gG).$$
    Here we have 
    $$\frac{\kappa_\alpha(x,y)}{1-\alpha}\leq \frac{2\textnormal{\diam}(\gG)}{d(x,y)}.$$
\end{proof}
Based on the concavity and upper-bounded property, we know that the Lin-Lu-Yau's Ricci curvature on a directed graph $ \gG = (\gV, \gE, w, d) $ is well-defined. 
Moreover, there exists an equivalent expression for the directed Ricci curvature that does not rely on limit operation. To introduce this formulation, we first define the gradient operator on directed graphs:

$$
\nabla_{xy} f := \frac{f(y) - f(x)}{d(x, y)} 
$$

Using the Laplacian operator defined earlier, we now give a limit-free expression for the Lin-Lu-Yau's Ricci curvature on directed graphs. 

\begin{theorem}[Limit-free $\kappa$]\label{thm:ozawa_equiv}
Let $ G = (V, E, w) $ be a directed graph. For any two distinct vertices $ x, y \in V $, the Lin-Lu-Yau's Ricci curvature from $ x $ to $ y $ satisfies:
$$
\kappa(x, y) = \inf_{f \in \mathcal{F}_{xy}} \nabla_{xy} \mathcal{L} f 
$$
where
$$
\mathcal{F}_{xy} := \left\{ f \in \lipv \mid \nabla_{xy} f = 1 \right\}.
$$
\end{theorem}
Note that $\mathcal{F}_{xy}$ is not compact. If we set a base point $y$ s.t. $f(y)=0$, then $\mathcal{F}_{xy}$ is compact since boundedness is straightforward, and closeness yields because when $\{f_n\}\in \mathcal{F}_{xy}$, $f_n\to f$ as $n\to\infty$, then $f(y)-f(x) = \lim_{n\to\infty}f_n(y)-\lim_{n\to\infty}f_n(x) = \lim_{n\to\infty}(f_n(y)-f_n(x))= d(x,y)$.

Münch and Wojciechowski firstly proposed this limit-free representation of Lin-Lu-Yau's Ricci curvature in~\cite{Munch_2019_Ollivier}, and Ozawa~\etal\;followed the same line to prove this representation in directed graphs~\cite{ozawa_geometric_2020}. 
We leverage the same technique as in~\cite{ozawa_geometric_2020} to prove this theorem, and carefully check that the proof still holds for our generalized definition of probability transition kernel using a balancing factor $\beta$.
Before proving this theorem, we first introduce a useful lemma from~\cite{ozawa_geometric_2020}(Lemma 3.9).
\begin{lemma}[\cite{ozawa_geometric_2020}]\label{lemma:ozawa_alpha}
Let $ G = (V, E, w) $ be a directed graph. For any pair of vertices $ x, y \in V $ with $ x \ne y $, we have:
\begin{equation}\label{eq:k_alpha_inf_lip1}
\frac{\kappa_\alpha(x, y)}{1 - \alpha} = \inf_{f \in \lipv} \left[ \frac{1}{1 - \alpha} (1 - \nabla_{xy} f) + \nabla_{xy} \mathcal{L} f \right] 
\end{equation}
\end{lemma}
\noindent By~\eqref{eq:W_dual} and~\eqref{eq:sum_f_mu}, the proof of this lemma is straightforward.  

\begin{proof}[Proof of \thmref{thm:ozawa_equiv}]\label{pf:thm_}
From Lemma \ref{lemma:ozawa_alpha}, we have:

\begin{align*}
    \frac{\kappa_\alpha(x,y)}{1 - \alpha} 
    &= \inf_{f \in \lipv} \left[ \frac{1}{1 - \alpha}(1 - \nabla_{xy} f) + \nabla_{xy} \mathcal{L} f \right] \\
    &\leq \inf_{f \in \mathcal{F}_{xy}} \left[ \frac{1}{1 - \alpha}(1 - \nabla_{xy} f) + \nabla_{xy} \mathcal{L} f \right]\\
    & = \inf_{f \in \mathcal{F}_{xy}} \nabla_{xy} \mathcal{L} f
\quad \text{(since } \nabla_{xy} f = 1)
\end{align*}
Taking $ \alpha \to 1 $, we obtain:
$$
\kappa(x,y) \leq \inf_{f \in \mathcal{F}_{xy}} \nabla_{xy} \mathcal{L} f
$$

To prove the opposite inequality, let $\Phi_{\alpha}(\cdot)$ be a functional defined by:
$$\Phi_\alpha(f):=\frac{1}{1-\alpha}\left(1-\nabla_{xy}f\right)+\nabla_{xy}\gL f.$$
Notice that for any constant $ c \in \mathbb{R} $, let $ g(x) = f(x) + c $ for any $x\in\gV$. Then:
\begin{align*}
    \Phi_\alpha(g) 
    & = \frac{1}{1-\alpha}\left(1-\nabla_{xy}g\right)+\nabla_{xy}\gL g \\
    & = \frac{1}{1-\alpha}\left(1-\nabla_{xy}(f + c)\right)+\nabla_{xy}\gL (f + c) \\
    & = \frac{1}{1-\alpha}\left(1-\nabla_{xy}f\right)+\nabla_{xy}\gL f \\
    & = \Phi_\alpha(f).
\end{align*}
This indicates that $\Phi_\alpha(\cdot)$ eliminates a constant drift for $f\in \lipv$. Recall
$$
\lipxv{x} := \{ f \in \lipv \mid f(x) = 0 \}
$$
is compact. Actually, $\lipxv{x}$ is a function space of $f:\gV\to \sR$ which eliminates a constant drift in $\lipv$ by $-f(x)$. Therefore, \eqref{eq:k_alpha_inf_lip1} holds in $\lipxv{x}$ as 
\begin{equation}\label{eq:k_alpha_inf_lip1x}
\frac{\kappa_\alpha(x,y)}{1 - \alpha}
= \inf_{f \in \text{Lip}_{1,x}(\gV)} \left[ \frac{1}{1 - \alpha}(1 - \nabla_{xy} f) + \nabla_{xy} \mathcal{L} f \right].
\end{equation}
Next, we are going to verify the continuity of $\Phi_\alpha (f)$ for fixed $x,y\in \gV$ on $\lipxv{x}$. Let the norm on $\lipxv{x}$ is $||f||_\infty :=\sup_{v\in\gV}|f(v)|$. For $f\in \lipxv{x}$, $f(x)=0$, thus
\begin{itemize}
    \item $\nabla_{xy}(\cdot)$ is continuous. For any $f, g \in \lipxv{x}$,  $|\nabla_{xy}f-\nabla_{xy}g| = \left\vert\frac{f(y)-f(x)}{d(x,y)}-\frac{g(y)-g(x)}{d(x,y)}\right\vert=\left\vert\frac{f(y)-g(y)}{d(x,y)}\right\vert\leq \frac{||f-g||_\infty}{d(x,y)}$.
    \item $\nabla_{xy}\gL (\cdot)$ is continuous. Recall $\gL$ is continuous, i.e., for any $f, g \in \lipxv{x}$, $|\gL f - \gL g|\leq C ||f-g||_\infty$ for come constant $C$. Thus $|\nabla_{xy}\gL f - \nabla_{xy}\gL g| \leq \frac{2C}{d(x,y)} ||f-g||_\infty.$
\end{itemize}
Hence $\Phi_\alpha(\cdot)$ is a continuous functional operator. 

By compactness of $\lipxv{x}$ and continuity of $\Phi_\alpha$, for a sequence of $\alpha\to 0$, according to Weierstrass theorem, there exists a sequence of $ \{f_{\alpha}\} \subset \text{Lip}_{1,x}(\gV) $ such that:
$$
f_{\alpha} = \arg\min_{f \in \text{Lip}_{1,x}(\gV)} \Phi_{\alpha} (f).
$$
Using the compactness of $\lipxv{x}$ again, we can find a convergent subsequence $ \{f_{\alpha_k}\}_{k=1}^\infty \subset \{f_{\alpha}\} $ (with $ \alpha_k \to 1 $ as $ k \to \infty $) such that
$$
\lim_{k \to \infty} f_{\alpha_k} = f^\ast \in \mathrm{Lip}_{1,x}(V).
$$
\noindent Now we take a closer look at~\eqref{eq:k_alpha_inf_lip1x}. 
On the left hand side of ~\eqref{eq:k_alpha_inf_lip1x}, $\frac{\kappa_\alpha(x,y)}{1 - \alpha} $ is monotone increasing and bounded above on $ \alpha \in [0, 1) $, therefore, 
$
\lim_{\alpha \to 1} \frac{\kappa_\alpha(x,y)}{1 - \alpha} \text{ exists}.
$
On the right hand side of~\eqref{eq:k_alpha_inf_lip1x}, because $f\in \lipxv{x}$ and $\gG$ is finite, then 
\begin{equation}\label{eq:f_bound}
||f||_\infty=\max_{u\in\gV}|f(u)| = \max_{u\in\gV} \left\{\max\left\{f(u), -f(u)\right\}\right\}
= \max_{u\in\gV} \left\{\max\left\{d(x,u), d(u,x)\right\}\right\}\leq \diam(\gG).
\end{equation}
Then $ | \nabla_{xy} \gL f |$ is bounded, i.e., 
\begin{equation}\label{eq:kappa_nabla_bound}
    \begin{aligned}
        |\nabla_{xy}\gL f| 
        & = \left|\frac{\gL f(y) - \gL f(x)}{d(x,y)}\right| \\
        & = \left|\frac{(I-\gP)f(y) - (I-\gP)f(x)}{d(x,y)}\right| \\
        & \leq \max_{u,v\in\gV} |1 + \gP(u,v)|\frac{ \left|f(y)\right|}{d(x,y)} \\
        & \leq \max_{u,v\in\gV} |1 + \gP(u,v)| \frac{||f||_\infty}{d(x,y)}\leq 2\frac{\diam(\gG)}{d(x,y)}.
    \end{aligned}
\end{equation}

As $\alpha$ approximates $1$, the existence of limit in the right hand side of~\eqref{eq:k_alpha_inf_lip1x} requires $  \frac{1}{1 - \alpha}(1 - \nabla_{xy} f_{\alpha})$ converging. Therefore $ \lim_{\alpha \to 1} (1 - \nabla_{xy} f_{1 - \alpha}) = 0 $, or equivalently, $ \lim_{\alpha \to 1} \nabla_{xy} f_{\alpha} = 1 $.
Thus by the continuity of $\nabla_{xy}(\cdot)$ operator, the following holds:
$$
\nabla_{xy} f^\ast = \nabla_{xy} \lim_{k \to \infty} f_{\alpha_k} = \lim_{k \to \infty} \nabla_{xy} f_{\alpha_k} = 1 \Rightarrow f^\ast \in \gF_{xy}.
$$
Also note that since $ f_{\alpha} $ is 1-Lipschitz, we have $ \nabla_{xy} f_{\alpha} \leq 1 $, and $\nabla_{xy}\gL (\cdot)$ is continuous, hence the following inequality holds:
$$
\kappa(x, y) = \lim_{\alpha \to 1} \left[ \frac{1}{1 - \alpha}(1 - \nabla_{xy} f_{\alpha}) + \nabla_{xy} \gL f_{\alpha} \right] 
\geq \lim_{k\to \infty} \nabla_{xy}\gL f_{\alpha_k}
= \nabla_{xy} \gL f^\ast \geq \inf_{f \in \gF_{xy}} \nabla_{xy} \gL f.
$$

Combining the conclusions from “$ \leq $” and “$ \geq $”, we obtain
$$
\kappa(x, y) = \inf_{f \in \gF_{xy}} \nabla_{xy} \gL f = \inf_{f \in \gF_{xy}} \nabla_{yx} \Delta f.
$$

\end{proof}

\section{Ricci Flow on Directed Graphs}\label{sec:directed_flow}
Given a directed graph $\gG=(\gV, \gE, \vw, d)$, and $\kappa_e(t)$ is the directed Lin-Lu-Yau's Ollivier Ricci curvature. Let $m=|\gE|$. Ricci flow deforms the graph weights $\vw(t)$ according to the Ricci curvature on edges, while the distance $d$ between vertex pairs remain unchanged. Formally, we define the Ricci flow as follow:
\begin{definition}[Ricci flow on directed graphs]\label{def:ricci_flow}
    For continuous time $t\in [0,+\infty)$, the edge weights on $\gG$ at time $t$ is denoted as $\vw(t)=(w_{e_1}(t), w_{e_2}(t), \dots, w_{e_m}(t))\in \sR^m_+$. Let $\vw^0=((w_{e_1}^0, w_{e_2}^0, \dots, w_{e_m}^0)$ and $\sum_{i=1}^mw_{e_i}^0=1$. Then the Ricci flow on directed graph $\gG$ is formulated as the following ODE:
    \begin{equation}\label{eq:continue_ricci_flow}
        \left\{\begin{array}{l}
        \vw(0) =\vw^0\\
        \frac{d w_e(t)}{dt}=-\kappa_e(t) w_e(t) + w_e(t)\sum_{h\in\gE}\kappa_h(t)w_h(t).
        \end{array}\right.
    \end{equation}
\end{definition}
\noindent \eqref{eq:continue_ricci_flow} is a scaling normalized Ricci flow of the unnormalized version:
\begin{equation}\label{eq:continue_ricci_flow_unnormalized}
    \left\{\begin{array}{l}
        \tilde{\vw}(0) =\tilde{\vw}^0\\
        \frac{d \tilde{w}_e(t)}{dt}=-\kappa_e(t) \tilde{w}_e(t).
        \end{array}\right.
\end{equation}
\noindent The volume of a graph is defined as $\vol(t)=\sum_{e\in\gE}w_e(t)$. Thus the normalized Ricci flow in \eqref{eq:continue_ricci_flow} preserves total volume, because $\sum_{i=1}^mw_{e_i}(0)=1$, and 
\begin{align*}
    \frac{d\vol(t)}{dt}& =\sum_{e\in\gE}\left[-\kappa_e(t) w_e(t) + w_e(t)\sum_{h\in\gE}\kappa_h(t)w_h(t)\right]\\
    & = -\sum_{e\in\gE}\kappa_e(t) w_e(t) + \sum_{e\in\gE}w_e(t)\sum_{h\in\gE}\kappa_h(t)w_h(t)\\
    & =0.
\end{align*}
The solution to \eqref{eq:continue_ricci_flow} (denoted as $\vw (t)=\{w_e(t)\}_{e\in\gE}$) and the solution of \eqref{eq:continue_ricci_flow_unnormalized} (denoted as $\tilde{\vw}(t)=\{\tilde{w}_e(t)\}_{e\in\gE}$) have a relation $w_e(t)=\tfrac{\tilde{w}_e(t)}{\sum_{h\in\gE}\tilde{w}_h(t)}$. The evolution process of curvature $\kappa_e$ are the same because it is not sensitive to scaling change. Readers may refer to~\cite{bai_Ollivier_2024} for more details.
Next, We prove existence and uniqueness of the solution of \eqref{eq:continue_ricci_flow}.

\begin{theorem}
     \eqref{eq:continue_ricci_flow} has unique solution $\vw(t)$ for $t\in [0,+\infty)$.
\end{theorem}
\begin{proof}
    Denote $F(\vw, t) := \Big( -\kappa_e(\vw, t) w_e(t) + w_e(t) \sum_{h\in\gE} \kappa_h(\vw, t) w_h(t) \Big)_{e\in \mathcal{E}}$ be a vector functional. Thus ~\eqref{eq:continue_ricci_flow} can be write as 
    $$\frac{d\vw}{dt} = F(\vw, t), \quad \vw(0) = \vw^0.$$
    \noindent\textbf{1. Short time existence.}
    By Cauchy-Lipschitz Theorem, to prove the existence and uniqueness of the solution to~\eqref{eq:continue_ricci_flow} for $t\in [0,T]$ where $0<T<+\infty$ , is to verify the local Lipschitz property of $F(\vw, t)$ w.r.t. $\vw$, and the continuity of $F(\vw, t)$ w.r.t. $t$. Notice that the continuity can be drawn from the Lipschitz property of $F$ w.r.t. $\vw$ and the continuity of $\vw(t)$. The second condition always holds by the ODE system in~\eqref{eq:continue_ricci_flow}. Therefore, we only need to prove the Lipschitz property.

    By the condition in~\defref{def:ricci_flow}, for any initial state $\vw^0$, there exists $\delta<\delta_0:=\min_{e\in \gE}w^0_e$ such that 
    $$
    \vw^0\in R_\delta := \{ (w_1, \dots, w_m) : \sum_{i=1}^m w_i = 1, \; w_i \geq \delta \}, 
    $$
    and $R_\delta\times [0,T]$ is close. Observing the right-hand side of~\eqref{eq:continue_ricci_flow}, proving the Lipschitz property of $F(\vw, t)$ suffices to prove the function $ \kappa_e w_e $ is locally Lipschitz for $\vw \in R_\delta $.

    Denote $h_e(\vw):=\kappa_e (\vw)w_e$. For any $\vw, \vw^\prime\in R_\delta$, there exists a constant $L_h$ such that 
    $|h_e(\vw) - h_e(\vw^\prime)|\leq L_h ||\vw-\vw^\prime||_\infty$, since 
    \begin{equation}\label{eq:h_lip}
        \begin{aligned}
            |h_e(\vw) - h_e(\vw^\prime)| 
            & = |\kappa_e (\vw)w_e - \kappa_e (\vw^\prime)w_e^\prime| \\
            & = |\kappa_e (\vw)w_e - \kappa_e (\vw^\prime)w_e + \kappa_e (\vw^\prime)w_e - \kappa_e (\vw^\prime)w_e^\prime| \\
            & \leq |\kappa_e (\vw) - \kappa_e (\vw^\prime)||w_e| + |\kappa_e (\vw^\prime)| |w_e - w_e^\prime| \\
            & \leq |\kappa_e (\vw) - \kappa_e (\vw^\prime)| + |\kappa_e (\vw^\prime)|\ ||\vw - \vw^\prime||_\infty. \quad (0\leq w_e\leq 1)
        \end{aligned}
    \end{equation}
    
    Next, we show that $\kappa_e(\vw)$ is Lipschitz, and $|\kappa_e(\vw)|$ is bounded for all $\vw\in R_\delta$.
    \begin{itemize}
        \item {\it $\kappa_e(\vw)$ is Lipschitz. } 
        Let $$g(f,\vw):= \nabla_{xy}\gL(f,\vw),$$ 
        $$\gF_{xy}^\ast :=\{f\in\gF_{xy}\mid f(x)=0\} = \{f\in\lipv\mid f(y)=d(x,y)\}.$$
        Then $\gF_{xy}^{\ast}$ depends on the distance function $d$ and the chioce of $x$ and $y$. $\gF_{xy}^\ast$ is compact as it fixed a base point to guarantee constant translation invariance, since for any constant $c\in \sR$, $\nabla_{xy}\gL (f+c)=\nabla_{xy}\gL (f)$.
        Thus by \thmref{thm:ozawa_equiv}, 
        $$\kappa_e(\vw)=\inf_{f\in\gF_{xy}}g(f,\vw) = \inf_{f\in\gF_{xy}^\ast}g(f,\vw).$$
        To prove the Lipschitz property of $\kappa_e(\vw)$, we first observe that
        \begin{equation}\label{eq:g_expand}
            \begin{aligned}
                g(f, \vw) 
                 & =\nabla_{xy}\gL(f,\vw) 
                 = \frac{1}{d(x,y)}\left[\gL f(y) - \gL f(x)\right] \\
                 & = \frac{1}{d(x,y)}\left[f(y)-f(x)-\sum_{u\in \gV} \gP_\vw (y,u)f(u) + \sum_{u\in \gV} \gP_\vw (x,u)f(u)\right].
            \end{aligned}
        \end{equation}
        We verify that $g(f, \vw)$ is uniformly Lipschitz for all fixed $f \in \gF_{xy}^\ast$, i.e., there exists $L_g>0$ such that $|g(f,\vw) - g(f, \vw^\prime)| \leq L_g ||\vw - \vw^\prime||_\infty$. It is followed by 
        \begin{equation}\label{eq:g_lip}
            \begin{aligned}
                |g(f, \vw) - g(f, \vw^\prime)| 
                & = \frac{1}{d(x,y)}\left[ -\sum_{u\in \gV} (\gP_\vw-\gP_{\vw^\prime}) (y,u)f(u) + \sum_{u\in \gV} (\gP_\vw-\gP_{\vw^\prime}) (x,u)f(u)  \right] \\
                & \leq \frac{||f||_\infty}{d(x,y)}\sum_{u\in \gV} \left[|(\gP_\vw-\gP_{\vw^\prime}) (y,u)| + |(\gP_\vw-\gP_{\vw^\prime}) (x,u)|\right] \\
                & \leq \frac{2L_{\gP}||f||_\infty}{d(x,y)} ||\vw - \vw^\prime||_\infty,
            \end{aligned}
        \end{equation}
        
        where $L_\gP= (\frac{1}{\delta}+\frac{1}{\delta^2})$. By definition of $\gP$ in~\eqref{eq:balance_ptk} and analyzing $P$(similarly $P^\prime$), for fixed $z, u\in\gV$, assume that $z\to u$, denote 
        $$a:=w_{zu},\quad a^\prime := w^\prime_{zu},\quad b:=\sum_{v:z\to v}w_{zv}\geq \deg^{out}_z \delta,\quad b^\prime:=\sum_{v:z\to v} w^\prime_{zv}\geq \deg^{out}_z \delta,$$
        then 
        \begin{align*}
            |(P_\vw-P_{\vw^\prime}) (z,u)| 
            & = \left|\frac{a}{b} - \frac{a^\prime}{b^\prime}\right| = \left|\frac{ab^\prime - a^\prime b^\prime + a^\prime b^\prime - a^\prime b}{bb^\prime} \right| \\
            & \leq \frac{|a-a^\prime|}{b} + \frac{|a^\prime|\ |b-b^\prime|}{b b^\prime}.
        \end{align*}
        Notice that $0\leq a, a^\prime \leq 1$, $b, b^\prime \geq \deg^{out}_z \delta$, and 
        $$|a-a^\prime| \leq ||\vw - \vw^\prime||_\infty,$$
        $$|b-b^\prime| = \left|\sum_{v:z\to v}w_{zv} - \sum_{v:z\to v}w^\prime_{zv}\right| \leq 
        \sum_{v:z\to v}\left|w_{zv} - w^\prime_{zv}\right| \leq \deg^{out}_x ||\vw -\vw^\prime||_\infty.$$
        Then 
        \begin{align*}
            |(P_\vw-P_{\vw^\prime}) (z,u)| 
             \leq \frac{1}{\deg^{out}_z}\left(\frac{1}{\delta} + \frac{1 }{\delta^2}\right)||\vw -\vw^\prime||_\infty.
        \end{align*}
        Similarly, $|(P^\prime_\vw-P^\prime_{\vw^\prime}) (z,u)| \leq \frac{1}{\deg^{in}_z}\left(\frac{1}{\delta} + \frac{1 }{\delta^2}\right)||\vw -\vw^\prime||_\infty$. 
        
        Thus by~\eqref{eq:balance_ptk}, we have
        \begin{align*}
            & \sum_{u\in \gV} |(\gP_\vw-\gP_{\vw^\prime}) (z,u)|  \\
            & \qquad = \sum_{u\in \gV} \left|\left((\beta(z) P_\vw+ (1-\beta(z)) P^\prime_\vw) - (\beta(z) P_{\vw^\prime}+ (1-\beta(z)) P^\prime_{\vw^\prime})\right) (z,u)\right| \\
            & \qquad = \sum_{u\in \gV} \left|\beta(z) (P_\vw - P_{\vw^\prime}) + (1-\beta(z)) (P^\prime_\vw - P^\prime_{\vw^\prime})  (z,u)\right| \\
            & \qquad \leq \beta(z) \sum_{u\in \gN^{out}(z)} \left| (P_\vw - P_{\vw^\prime})(z,u)\right| + (1-\beta(z))  \sum_{u\in \gN^{in}(x)} \left|(P^\prime_\vw - P^\prime_{\vw^\prime})  (z,u)\right| \\
            & \qquad \leq \beta(z) \sum_{u\in \gN^{out}(x)} \frac{1}{\deg^{out}_z}\left(\frac{1}{\delta} + \frac{1 }{\delta^2}\right)||\vw -\vw^\prime||_\infty + (1-\beta(z))  \sum_{u\in \gN^{in}(x)} \frac{1}{\deg^{in}_z}\left(\frac{1}{\delta} + \frac{1 }{\delta^2}\right)||\vw -\vw^\prime||_\infty \\
            & \qquad \leq \beta(z) \left(\frac{1}{\delta} + \frac{1 }{\delta^2}\right)||\vw -\vw^\prime||_\infty + (1-\beta(z))  \left(\frac{1}{\delta} + \frac{1 }{\delta^2}\right)||\vw -\vw^\prime||_\infty \\
            & \qquad = \left(\frac{1}{\delta} + \frac{1 }{\delta^2}\right)||\vw -\vw^\prime||_\infty := L_\gP ||\vw -\vw^\prime||_\infty.
        \end{align*}
        Notice that $f\in\gF_{xy}^\ast$, which implies that $||f||_\infty\leq \diam (\gG)$ by~\eqref{eq:f_bound}.
        Then for $\forall f\in \gF_{xy}^\ast$, ~\eqref{eq:g_lip} writes
        \begin{equation}\label{eq:g_lip2}
        |g(f, \vw) - g(f, \vw^\prime)| \leq \frac{2L_{\gP}\diam (\gG)}{d(x,y)} ||\vw - \vw^\prime||_\infty := L_g ||\vw - \vw^\prime||_\infty.
        \end{equation}

        Finally, for $\kappa_e(\vw) = \inf_{f\in\gF_{xy}^\ast} g(f\vw)$, we show that $|\kappa_e (\vw) - \kappa_e(\vw^\prime)| \leq L_g ||\vw- \vw^\prime||_\infty$. 
       Since $g(f, \vw)$ is continuous for $f$ in the compact set $\gF_{xy}^\ast$, for any $\epsilon > 0$, there exists $f_\epsilon \in \gF_{xy}^\ast$ such that $g(f_\epsilon, \vw) < \kappa_e(\vw) + \epsilon$. Thus by~\eqref{eq:g_lip2}, 
        $$\kappa_e(\vw^\prime) \leq g(f_\epsilon, \vw^\prime) \leq g(f_\epsilon, \vw) + L_g ||\vw-\vw^\prime||_\infty < \kappa_e(\vw) + \epsilon + L_g ||\vw-\vw^\prime||_\infty.$$
        As $\epsilon\to 0$, $\kappa_e(\vw^\prime) \leq \kappa_e(\vw) + L_g ||\vw-\vw^\prime||_\infty$. By exchanging $\vw$ and $\vw^\prime$, we get the opposite inequality $\kappa_e(\vw) \leq \kappa_e(\vw^\prime) + L_g ||\vw-\vw^\prime||_\infty$. Thus 
        \begin{equation}\label{eq:kappa_lip}
            |\kappa_e(\vw)-\kappa_e(\vw^\prime)| \leq L_g ||\vw-\vw^\prime||_\infty, 
        \end{equation}
        where $L_g=\frac{2L_{\gP}\diam (\gG)}{d(x,y)}$.
        \item {\it $|\kappa_e(\vw)|$ is bounded.} As showed in~\eqref{eq:kappa_nabla_bound}  and~\thmref{thm:ozawa_equiv}, 
         \begin{equation}\label{eq:kappa_bound}
             |\kappa_e(\vw)| =\left|\inf_{f\in \gF_{xy}^\ast} \nabla_{xy}\gL f\right|\leq \sup_{f\in \gF_{xy}^\ast} \left|\nabla_{xy}\gL f\right| \leq \frac{2\diam(\gG)}{d(x,y)}.
         \end{equation}
    \end{itemize}
    \eqref{eq:h_lip} writes 
    $$|h_e(\vw) - h_e(\vw^\prime)| \leq \frac{2\diam(\gG)}{d(x,y)} (L_\gP +1)  ||\vw - \vw^\prime||_\infty := L_h ||\vw - \vw^\prime||_\infty,$$
    where $L_h = \frac{2\diam(\gG)}{d(x,y)} (\frac{1}{\delta} +\frac{1}{\delta^2} +1)$. 
    
    Thus $F(\vw, t)$ is continuous w.r.t. $t$, and local Lipschitz w.r.t. $\vw$. There exists a unique solution for ODE system in~\eqref{eq:continue_ricci_flow} for $t\in [0,T]$.

    \noindent\textbf{2. Long time existence.}
    According to the Extension Theorem~\cite{Zhang_2006_Ordinary}, let $$T^\ast = \sup_{T>0} \{\eqref{eq:continue_ricci_flow} \text{ has unique solution in } [0, T]\},$$
    and $[0,T^\ast)$ be the maximal interval of existence, then $T^\ast < +\infty$ only if $|\vw(t)| \to 0$ or $|\vw(t)|\to +\infty$ as $t\to T^{\ast-}$. 
    Let $d_{min}=\min_{u,v\in\gE}d(u,v)$ and $d_{max}=\max_{u,v\in\gE}d(u,v)=\diam(G)$. For any $e=(x,y)\in\gE$, according to the boundedness of $\kappa_e$ in~\eqref{eq:kappa_bound}, we have
    $$-\tfrac{2d_{max}}{d_{min}}\leq -\tfrac{2\diam(G)}{d(x,y)}\leq \kappa_e \leq \tfrac{2\diam(G)}{d(x,y)} \leq \tfrac{2d_{max}}{d_{min}}.$$
    Then then evolution of edge weight in~\eqref{eq:continue_ricci_flow} is bounded by:
    \begin{align*}
        \frac{dw_e(t)}{dt} 
        & \geq -\frac{2d_{max}}{d_{min}}w_e(t)-w_e(t)\sum_{h\in \gE}\frac{2d_{max}}{d_{min}}w_h(t) \\
        & = -\frac{2d_{max}}{d_{min}}\left(1 + \sum_{h\in\gE}w_h(t)\right)w_e(t) \\
        & = -\frac{4d_{max}}{d_{min}}w_e(t),
    \end{align*}
    and similarly 
    \begin{align*}
        \frac{dw_e(t)}{dt} 
        \leq \frac{2d_{max}}{d_{min}}w_e(t)+w_e(t)\sum_{h\in \gE}\frac{2d_{max}}{d_{min}}w_h(t) = \frac{4d_{max}}{d_{min}}w_e(t),
    \end{align*}
    which means $\vw_e(t)$ is bounded by:
    $$w_e(0)e^{-\frac{4d_{max}}{d_{min}}}\leq w_e(t)\leq w_e(0)e^{\frac{4d_{max}}{d_{min}}}.$$
    Here we have $T^\ast=+\infty$, hence~\eqref{eq:continue_ricci_flow} has unique solution for $t\in [0,+\infty).$
    
\end{proof}

\section{Discrete Ricci flow on directed graphs}\label{sec:discrete_flow}

\subsection{Limit-free Curvature}
The Lin–Lu–Yau (LLY) curvature has attracted significant attention in recent years due to its ability to capture geometric and functional properties of complex networks. In particular, the limit-free formulation of LLY curvature has been extensively applied to undirected graphs~\cite{bai_Ollivier_2024, Lai2022Normalized, Lai2023Deeper}, providing efficient and practical tools for curvature-based analysis without relying on limiting processes. Motivated by these developments, and in order to facilitate computations, we extend a limit-free definition of curvature to directed graphs, which is similar to the undirected version in~\cite{bai2020sum}.
\begin{definition}[$\ast$-coupling on directed graphs]\label{def:limit_free_curv}
    Given a directed graph $\gG=(\gV,\gE, \vw, d)$, for any $z\in \gV$, denote $\mu_z^0$ as $\mu_z$. For any distinct nodes $x, y\in \gV$, coupling $B$ is called an $\ast$-coupling if it satisfies:
    \begin{enumerate}
        \item $B(x,y)>0$, and $B(u, v)\leq 0$ if $u\neq x$ or $v\neq y$.
        \item $\sum_{u, v\in\gV}B(u, v)=0$.
        \item $\sum_{v\in \gV}B(u, v)=-\mu_x(u)$ for $\forall u\neq x$.
        \item $\sum_{v\in \gV}B(u, v) = -\mu_y(v)$ for $\forall v\neq y$.
    \end{enumerate}
    Moreover, $\ast$-coupling based Lin-Lu-Yau Ricci curvature is defined as:
    \begin{equation}\label{eq:limit_free_kappa}
    \kappa^\ast(x,y)=\frac{1}{d(x,y)}\sup_{B}\sum_{u, v\in\gV}B(u, v)d(u,v).
    \end{equation}
\end{definition}
Similar to \cite{bai2020sum}, $\kappa^\ast$ in~\eqref{eq:limit_free_kappa} is a limit-free dual of $\kappa$ in~\defref{def:LLY_def}. 
Although the distance is directed and the definition of probability distribution in~\defref{def:prob_def} is different, the proof of Theorem 2.14 in~\cite{bai2020sum} still holds.

\subsection{Discrete Ricci Flow}
While the continuous Ricci flow on directed graphs is formulated as the system of ODEs in Definition \defref{def:ricci_flow}, for practical implementations, especially in computational settings, we often need a discrete-time approximation. As the discretization process in~\cite{Lai2022Normalized}, the idea is to update edge weights iteratively in each iteration by small steps, following the gradient induced by the curvature.

\begin{definition}[Discrete Ricci Flow on Directed Graphs]\label{def:discrete_ricci_flow}
Let $\gG=(\gV,\gE, \vw, d)$ be a finite directed graph with positive edge weights $\vw=(w_e)_{e\in\gE}$ normalized such that $\sum_{e\in\gE} w_e = 1$. Given an initial weight vector $\vw^0$, the discrete Ricci flow is defined for a sequence of discrete time $t=0,1,2, \dots, T$  with step size $s > 0$ by the iteration:
\begin{equation}\label{eq:discrete_ricci_flow}
    w_e^{t+1} = w_e^t - w_e^t \cdot s \Big(\kappa_e^t - \overline{\kappa}^t\Big), \qquad e\in\gE,
\end{equation}
where
\begin{equation*}
    \overline{\kappa}^t = \sum_{h\in\gE} \kappa_h^t\, w_h^t
\end{equation*}
is the weighted average curvature at time step $t$. 
\end{definition}

\begin{remark}
The update rule~\eqref{eq:discrete_ricci_flow} is an explicit discretization of the continuous Ricci flow ODE~\eqref{eq:continue_ricci_flow}. 
\begin{itemize}
\item It ensures positivity of the weights, i.e., $w_e^{t+1}\ge0$ iff
$
1 - s\big(\kappa_e^t-\overline{\kappa}^t\big)\;\ge\;0 $
for all $e\in\gV$.
In particular, according to~\eqref{eq:kappa_bound}, denote 
$$K:= \frac{2\diam(\gG)}{\min_{e=(x,y)\in \gE}\{d(x,y)\}},$$ 
then $|\kappa_e^t|\le K$ uniformly at time $t$. 
Thus $\kappa_e^t-\overline{\kappa}^t\le 2K$ and any $s\le \frac{1}{2K}$ preserves nonnegativity of $w_e^{t+1}$.
(If $\kappa_e^t-\overline{\kappa}^t\le 0$ the positivity of $w_e^{t+1}$ always holds).

\item The volume of the graph still preserves as $\sum_{e\in\gV}w_e^{t+1}=1$, making it acts as a proper weight in the average curvature $\overline{\kappa}^t$.
\end{itemize}
\end{remark}

\begin{remark}
In~\defref{def:discrete_ricci_flow}, we assumed the total edge weight (``volume'') is normalized to $1$, i.e.\ $\sum_{e\in\gE} w_e = 1$.  This normalization is \emph{not essential}.  One may equivalently work with any fixed total volumne
$\sigma \;:=\; \sum_{e\in\gE} w_e \;>\; 0,$
provided the average curvature is defined with the same normalization. Concretely, put
$$\overline{\kappa}_\sigma(t) \;:=\; \frac{1}{\sigma}\sum_{h\in\gE}\kappa_h(t)\,w_h(t).$$
Then the modified continuous flow preserving total volumne $\sigma$ reads
\begin{equation}\label{eq:continue_ricci_flow_sigma}
\frac{d w_e(t)}{dt}
= -\kappa_e(t)\,w_e(t) + \frac{w_e(t)}{\sigma}\sum_{h\in\gE}\kappa_h(t)\,w_h(t)
= w_e(t)\big(-\kappa_e(t)+\overline{\kappa}_\sigma(t)\big).
\end{equation}
The discrete \eqref{eq:discrete_ricci_flow} is modified analogously:
\begin{equation}\label{eq:discrete_ricci_flow_sigma}
w_e^{t+1} \;=\; w_e^t - s\, w_e^t\big(\kappa_e^t - \overline{\kappa}_\sigma^t\big),
\qquad
\overline{\kappa}_\sigma^t := \frac{1}{\sigma}\sum_{h\in\gE}\kappa_h^t\, w_h^t.
\end{equation}
The solution to \eqref{eq:continue_ricci_flow_sigma} (denoted as $\vw (t)=\{w_e(t)\}_{e\in\gE}$) and the solution of \eqref{eq:continue_ricci_flow_unnormalized} (denoted as $\tilde{\vw}(t)=\{\tilde{w}_e(t)\}_{e\in\gE}$) have a relation $w_e(t)=\tfrac{\sigma}{\sum_{h\in\gE}\tilde{w}_h(t)}\tilde{w}_e(t)$, since 
\begin{align*}
    \frac{dw_e(t)}{dt}
    & = \sigma\cdot \frac{d\frac{\tilde{w}_e(t)}{\sum_{h}\tilde{w}_h(t)}}{dt}=\sigma\cdot\left(\frac{d\tilde{w}_e(t)}{\sum_{h}\tilde{w}_h(t)dt} -\frac{\tilde{w}_e(t)\sum_{h}\frac{d\tilde{w}_h(t)}{dt}}{\left(\sum_{h}\tilde{w}_h(t)\right)^2}\right)  \\
    & = \frac{\sigma }{\sum_{h}\tilde{w}_h(t)} \left(-\kappa_e(t)\tilde{w}_e(t)\right)+\frac{\sigma \tilde{w}_e(t)}{\left(\sum_{h}\tilde{w}_h(t)\right)^2}\sum_{h}\kappa_h(t)\tilde{w}_h(t)\\
    & = -\kappa_e(t)w_e(t) + \frac{w_e(t)}{\sigma}\sum_{h}\kappa_h(t)w_h(t).
\end{align*}
where $\kappa_e(t)$ are the same in both equations by scaling change of weights for any $t\geq 0$, and the volume $\sum_{e}w_e(t) = \sum_{e}\tfrac{\sigma}{\sum_{h\in\gE}\tilde{w}_h(t)}\tilde{w}_e(t)=\sigma$ for all $t\geq 0$.
\end{remark}

\section{Examples and numerical study}\label{sec:exam_app}

\begin{example}[Directed Ricci Flow on Cycles]\label{ex:cycle}
Consider a directed cycle graph $\gG=(\gV,\gE, \vw, d)$ with $n$ vertices $\gV=\{v_1,v_2,\dots, n\}$ and edges
$$\gE=\{e_1=(v_1\to v_2),\; e_2=(v_2\to v_3),\; \dots,\; e_{n-1}=(v_{n-1}\to v_n), \; e_n=(v_n\to v_1)\}.$$
Initialize the distance on edges be $d_{e_i}=1$ for $i=1,\dots, n$, and the edge weights equally as 
$$\vw^0=(w_{e_1}^0,w_{e_2}^0,\dots, w_{e_n}^0)=(\tfrac{\sigma}{n},\tfrac{\sigma}{n},\dots,\tfrac{\sigma}{n}).$$
The balance factor $\beta(v_i)$ in~\eqref{eq:gP_strong_conn} is set as a constant $\beta\in[0,1]$.

\noindent \textbf{Case n=2.} See~\figref{fig:2_cycle}.
\begin{figure}[H]
    \begin{center}
        \begin{tikzpicture}[->,>=stealth',shorten >=1pt,auto,node distance=2.5cm, thick]
            \tikzstyle{vertex}=[circle,draw,fill=blue!15,minimum size=16pt,inner sep=0pt]
            \node[vertex] (v1) {$v_1$};
            \node[vertex] (v2) [right of=v1] {$v_2$};
            \path (v1) edge [bend left] node {$w_{e_1}^t\equiv w_{e_1}^0$} (v2)
                  (v2) edge [bend left] node {$w_{e_2}^t\equiv w_{e_2}^0$} (v1);
        \end{tikzpicture}
    \end{center}
    \caption{Directed 2-Cycle.}
    \label{fig:2_cycle}
\end{figure}
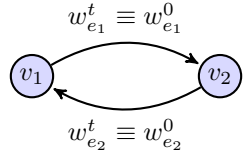
For $\alpha\in(\tfrac{1}{2},1)$, by~\eqref{eq:mu} and~\eqref{eq:balance_ptk}, the probability distributions on $\gV$ are:
$$
\mu_{v_1}^\alpha(z)=
    \begin{cases}
        \alpha, & z=v_1;\\
        (1-\alpha), & z=v_2,
    \end{cases}
    \qquad
    \mu_{v_2}^\alpha(z)=
    \begin{cases}
        (1-\alpha), & z=v_1;\\
        \alpha, & z=v_2.
    \end{cases}
$$
The optimal transport plan sends $(2\alpha - 1)$ from $v_1$ to $v_2$ for $e_1$ or from $v_2$ to $v_1$ for $e_2$, so
$$W_1(\mu_{v_1}^\alpha,\mu_{v_2}^\alpha)=(2\alpha - 1)\cdot d(v_1, v_2) = (2\alpha - 1), \qquad W_1(\mu_{v_2}^\alpha,\mu_{v_1}^\alpha)=(2\alpha - 1)\cdot d(v_2, v_1) = (2\alpha - 1).$$
By~\defref{def:LLY_def},
$$\kappa_\alpha(v_1,v_2)=1-\frac{W_1(\mu_{v_1}^\alpha,\mu_{v_2}^\alpha)}{d(v_1,v_2)}=1-\frac{2\alpha - 1}{1}=2(1 - \alpha).$$

Similarly, $\kappa_\alpha(v_2,v_1)=2(1 - \alpha)$. Then
$$
\kappa_{e_1}=\kappa_{e_2}=\lim_{\alpha\to 1}\tfrac{\kappa_\alpha}{1-\alpha}=2, \qquad \overline{\kappa}_\sigma = 2.
$$
Hence 
$$\frac{d w_{e_i}(t)}{dt}
= w_{e_i}(t)\big(-\kappa_{e_i}(t)+\overline{\kappa}_\sigma(t)\big) = 0\qquad \text{for }i = 1,2.$$
so the weights remain unchanged in this case.

\noindent \textbf{Case n=3.} See~\figref{fig:3_cycle}.
\begin{figure}[t]
    \centering
    \begin{center}
    \begin{tikzpicture}[->,>=stealth',shorten >=1pt,auto,node distance=2cm, thick]
      \tikzstyle{vertex}=[circle,draw,fill=blue!12,minimum size=20pt,inner sep=0pt]
      \begin{scope}[xshift=-3cm]
        \node[vertex] (v1) at (-1, 1) {$v_1$};
        \node[vertex] (v2) at (1, 1) {$v_2$};
        \node[vertex] (v3) at (0, -1.2) {$v_3$};
        
        \path (v1) edge [bend left] node {$w_{e_1}^0=1$} (v2)
              (v2) edge [bend left] node {$w_{e_2}^0=1$} (v3)
              (v3) edge [bend left] node {$w_{e_3}^0=1$} (v1);
        \node at (0,-2) {\small (a) Initialization};
      \end{scope}
    
      \begin{scope}[xshift=3cm]
        \node[vertex] (v1) at (-1, 1) {$v_1$};
        \node[vertex] (v2) at (1, 1) {$v_2$};
        \node[vertex] (v3) at (0, -1.2) {$v_3$};
        
        \path (v1) edge [bend left] node {$w_{e_1}^t\equiv 1$} (v2)
              (v2) edge [bend left] node {$w_{e_2}^t\equiv 1$} (v3)
              (v3) edge [bend left] node {$w_{e_3}^t\equiv 1$} (v1);
        \node at (0,-2) {\small (b) Evolution while $\beta(v_1)=\beta(v_2)= \beta(v_3)$};
      \end{scope}
    \end{tikzpicture}
    \end{center}
    \caption{Directed 3-Cycle.}
    \label{fig:3_cycle}
\end{figure}
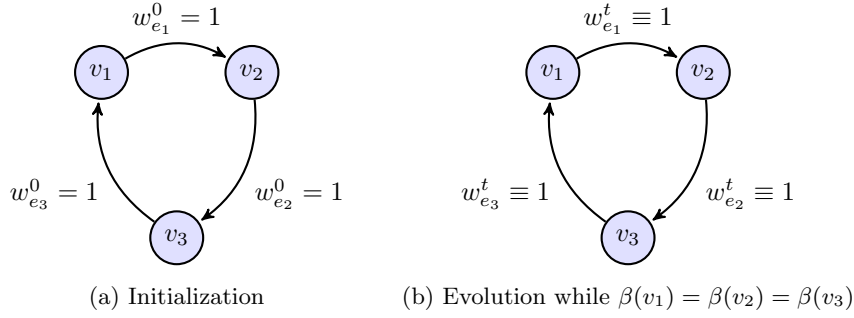
\vspace{4pt}
The probability distributions on $\gV$ are:
$$
\mu_{v_1}^\alpha(z)=
    \begin{cases}
        \alpha, & z=v_1;\\
        (1-\alpha)\beta, & z=v_2;\\
        (1-\alpha)(1-\beta), & z=v_3,
    \end{cases}
\qquad
\mu_{v_2}^\alpha(z)=
    \begin{cases}
        (1-\alpha)(1-\beta), & z=v_1;\\
        \alpha, & z=v_2;\\
        (1-\alpha)\beta, & z=v_3,
    \end{cases}
$$
$$
\mu_{v_3}^\alpha(z)=
    \begin{cases}
        (1-\alpha)\beta, & z=v_1;\\
        (1-\alpha)(1-\beta), & z=v_2;\\
        \alpha, & z=v_3.
    \end{cases}
$$

By symmetry in the cycle graph, for $(v_i, v_j)\in \gE$, the optimal transport distances are:
\begin{equation*}
    W_1(\mu_{v_i}^\alpha,\mu_{v_j}^\alpha)=\left\{
    \begin{aligned}
        &
        1 - 3\beta (1-\alpha),  
        & \quad 0\leq \beta < 0.5; \\
        &
        \alpha - \tfrac{1}{2}(1-\alpha), 
        & \quad \beta = 0.5;\\
        &
        \alpha + (1-\alpha)(3\beta - 2), 
        & \quad 0.5< \beta \leq 1.
    \end{aligned}
    \right.
\end{equation*}
So by~\defref{def:LLY_def}, for each $(v_i,v_j)\in\gE$:
\begin{equation*}
    \kappa_\alpha(v_i,v_j)=1-\frac{W_1(\mu_{v_i}^\alpha,\mu_{v_j}^\alpha)}{d(v_i,v_j)}=\left\{
    \begin{aligned}
        &3\beta(1-\alpha),  
        && 0\leq \beta < 0.5; \\
        & \tfrac{3}{2}(1-\alpha), 
        && \beta = 0.5;\\
        & 
        3 (1-\beta)(1-\alpha), 
        && 0.5< \beta \leq 1,
    \end{aligned}
    \right.
\end{equation*}
\begin{equation*}
    \kappa(v_i,v_j)=\lim_{\alpha\to 1}\frac{\kappa_{\alpha}(v_i,v_j)}{1-\alpha}=\left\{
    \begin{aligned}
        &3\beta,  
        && 0\leq \beta < 0.5; \\
        & \tfrac{3}{2}, 
        && \beta = 0.5;\\
        & 3 (1-\beta), 
        && 0.5< \beta \leq 1.
    \end{aligned}
    \right.
\end{equation*}

Then the average curvature $\overline{\kappa}_\sigma = \kappa(v_i,v_j)$.
Hence the evolution of weights satisfies
$$
\frac{d w_{e_i}(t)}{dt}
= w_{e_i}(t)\big(-\kappa_{e_i}(t)+\overline{\kappa}_\sigma(t)\big)=0,\qquad \text{for }i=1,2,3.
$$
So the weights remain unchanged under the directed Ricci flow process.

\end{example}
Actually, for $n$-cycle ($n\ge 3$) graph satisfies the initial conditions in~\exref{ex:cycle}, the probability distributions on $v_i\in\gV$ writes:
$$
\mu_{v_i}^\alpha(z)=
    \begin{cases}
        \alpha, & z=v_i;\\
        (1-\alpha)\beta, & z:v_i\to z;\\
        (1-\alpha)(1-\beta), & z: z\to v_i; \\
        0, & otherwise.
    \end{cases}
$$
Similar to the computation in~\exref{ex:cycle} (case $n=3$), weights on edges in $n$-cycle remain unchanged. But if we specify the balancing factor for each node as $\beta(v_i)$, the evolution of edge weights is shown in~\figref{fig:3cycle_varybeta}. If a node attends more to the out-flow, the weight on the out-edges decreases and vice versa. 
\begin{figure}[htbp]
    \centering
    \begin{center}
    \begin{tikzpicture}[->,>=stealth',shorten >=1pt,auto,node distance=2cm, thick]
      \tikzstyle{vertex}=[circle,draw,fill=blue!12,minimum size=20pt,inner sep=0pt]
      \begin{scope}[xshift=-3cm]
        \node[vertex] (v1) at (-1, 1) {$v_1$};
        \node[vertex] (v2) at (1, 1) {$v_2$};
        \node[vertex] (v3) at (0, -1.2) {$v_3$};
        
        \path (v1) edge [bend left] node {$w_{e_1}^t${\color{red}$\uparrow$}} (v2)
              (v2) edge [bend left] node {$w_{e_2}^t${\color{blue}$\downarrow$}} (v3)
              (v3) edge [bend left] node {$w_{e_3}^t${\color{blue}$\downarrow\downarrow$}} (v1);
        \node at (-1.5, 1.9) {$\beta(v_1)=0.2$};
        \node at (1.3, 1.9) {$\beta(v_2)=0.4$};
        \node at (0, -2) {$\beta(v_3)=0.7$};

      \end{scope}
    
      \begin{scope}[xshift=3cm]
        \node[inner sep=0pt] (russell) at (1,0){\includegraphics[width=0.6\textwidth]{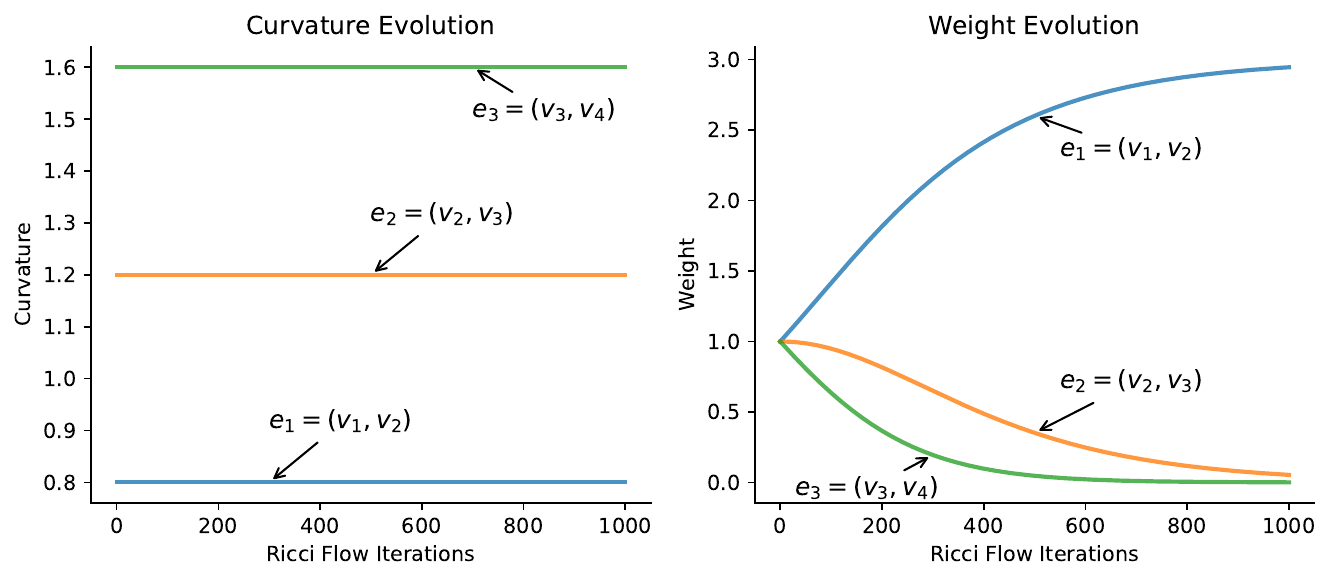}};
        \node at (1,-2.8) {\small Evolution of curvature and weight on each edge.};
      \end{scope}
    \end{tikzpicture}
    \end{center}
    \caption{Illustration of the effect about balance factor $\beta(x)$ in 3-cycles. Set $\beta(v_1)=0.2,\; \beta(v_2)=0.4,\; \beta(v_3)=0.7$. Red arrows indicate increased weight, blue arrows indicate decreased weight.}
    \label{fig:3cycle_varybeta}
\end{figure}
\vspace{4pt}
\begin{example}[Non-symmetric triangle]\label{ex:nonsymmetric_triangle}
Consider the directed (non-symmetric) triangle (see~\figref{fig:nonsymmetric_triangle}) with
$$
\gV=\{v_1,v_2,v_3\},\qquad
\gE=\{e_1=(v_1\to v_2),\; e_2=(v_2\to v_3),\; e_3=(v_3\to v_1),\; e_4=(v_1\to v_3)\}.
$$
We set graph distances $d(u,v)=1$ for all edges $e=(u,v)\in \gE$ and initialize edge weights uniformly. 
Fix $\alpha\in(\tfrac{1}{2},1)$, and let $\beta$ be the balancing factor.
\begin{figure}[H]
    \centering
    \begin{center}
    \def\scale{0.8}
    \begin{tikzpicture}[->,>=stealth',shorten >=1pt,auto,node distance=2cm, thick]
      \tikzstyle{vertex}=[circle,draw,fill=blue!12,minimum size=20pt,inner sep=0pt]
      \node[vertex] (v1) at (0, 0) {$v_1$};
      \node[vertex] (v2) at (2 * \scale, 3.4 * \scale) {$v_2$};
      \node[vertex] (v3) at (4 * \scale, 0) {$v_3$};
      \draw[->,bend left] (v1) to node[midway,left] {$e_1$} (v2);
      \draw[->,bend left] (v2) to node[midway,right] {$e_2$} (v3);
      \draw[->,bend left] (v3) to node[midway,below] {$e_3$} (v1);
      \draw[->,bend left] (v1) to node[midway,above] {$e_4$} (v3);
      \node at (0, -1.2 * \scale) {$\beta(v_1)$};
      \node at (2 * \scale, 4.2 * \scale) {$\beta(v_2)$};
      \node at (4 * \scale, -1.2 * \scale) {$\beta(v_3)$};
    \end{tikzpicture}
    \end{center}
    \caption{Non-sysmetric triangle.}
    \label{fig:nonsymmetric_triangle}
\end{figure}
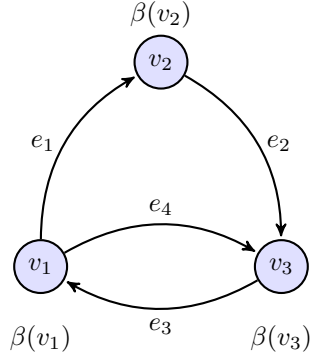
\vspace{4pt}
\noindent Recall
$$
\mu_x^\alpha(x)=\alpha,\qquad \mu_x^\alpha(z)=(1-\alpha)\,\gP(x,z)\ (z\ne x),
$$
with $\gP(x,z)=\beta(x)P(x,z)+(1-\beta(x))P'(x,z)$, and the out / in probability transition kernel
$P(x,z)=\tfrac{w_{xz}}{\sum_{u:x\to u}w_{xu}}$ for $z\in\mathcal N^{out}_x$ and
$P'(x,z)=\tfrac{w_{zx}}{\sum_{u:x\to u}w_{ux}}$ for $z\in\mathcal N^{in}_x$, we get:

$
\begin{aligned}
\bm{\mu_{v_1}^\alpha} &: \quad
\mu_{v_1}^\alpha(v_1)=\alpha,\\
&\qquad \mu_{v_1}^\alpha(v_2)=(1-\alpha)\Big(\beta(v_1)\cdot\tfrac{w_{e_1}}{w_{e_1}+ w_{e_4}} + (1-\beta(v_1))\cdot 0\Big)
=(1-\alpha)\beta(v_1)\tfrac{w_{e_1}}{w_{e_1}+ w_{e_4}},\\
&\qquad \mu_{v_1}^\alpha(v_3)=(1-\alpha)\Big(\beta(v_1)\cdot\tfrac{w_{e_4}}{w_{e_1}+ w_{e_4}} + (1-\beta(v_1))\cdot 1\Big)
=(1-\alpha)\left[1-\beta(v_1)\tfrac{w_{e_1}}{w_{e_1}+ w_{e_4}}\right].
\end{aligned}
$

\medskip
$
\begin{aligned}
\bm{\mu_{v_2}^\alpha} &: \quad
 \mu_{v_2}^\alpha(v_1)=(1-\alpha)(1-\beta(v_2)),\\
& \qquad \mu_{v_2}^\alpha(v_2)=\alpha,\\
&\qquad \mu_{v_2}^\alpha(v_3)=(1-\alpha)\,\beta(v_2)\cdot 1=(1-\alpha)\beta(v_2).
\end{aligned}
$

\medskip
$
\begin{aligned}
\bm{\mu_{v_3}^\alpha} &: \quad
,\mu_{v_3}^\alpha(v_1)=(1-\alpha)\Big(\beta(v_3)\cdot 1 + (1-\beta(v_3))\cdot\tfrac{w_{e_4}}{w_{e_2}+ w_{e_4}}\Big)
=(1-\alpha)\left[1 - (1-\beta(v_3))\tfrac{w_{e_1}}{w_{e_1}+ w_{e_4}}\right]\\
&\qquad \mu_{v_3}^\alpha(v_2)=(1-\alpha)\cdot(1-\beta(v_3))\tfrac{w_{e_2}}{w_{e_2}+ w_{e_4}}
,\\
&\qquad \mu_{v_3}^\alpha(v_3)=\alpha.
\end{aligned}
$

\noindent Let $\vw^0 = (w_{e_1}^0,w_{e_2}^0, w_{e_3}^0, w_{e_4}^0)=(1,1,1,1).$
Then the distributions at $t=0$ become:

$$
\begin{aligned}
\mu_{v_1}^\alpha &= (\,\alpha,\; (1-\alpha)\tfrac{\beta(v_1)}{2},\; (1-\alpha)\big(1-\tfrac{\beta(v_1)}{2}\big)\,)
,\\[3pt]
\mu_{v_2}^\alpha &= (\,(1-\alpha)(1-\beta(v_2)),\; \alpha,\; (1-\alpha)\beta(v_2)\,)
,\\[3pt]
\mu_{v_3}^\alpha &= \Big((1-\alpha)\tfrac{1+\beta(v_3)}{2},\; (1-\alpha)\tfrac{1-\beta(v_3)}{2},\; \alpha\Big).
\end{aligned}
$$
\begin{figure}[htbp]
    \centering
    \begin{center}
    \begin{tikzpicture}[->,>=stealth',shorten >=1pt,auto,node distance=2cm, thick]
      \tikzstyle{vertex}=[circle,draw,fill=blue!12,minimum size=20pt,inner sep=0pt]
      \def\scale{0.7}
      \begin{scope}[xshift=-6cm]
          \node[vertex] (v1) at (0,0) {$v_1$};
          \node[vertex] (v2) at (2 * \scale,3.4 * \scale) {$v_2$};
          \node[vertex] (v3) at (4 * \scale,0) {$v_3$};
          \draw[->,bend left, OliveGreen] (v1) to node[midway,left] {$w_{e_1}\uparrow$} (v2);
          \draw[->,bend left, OliveGreen] (v2) to node[midway,right] {$w_{e_2}\uparrow$} (v3);
          \draw[->,bend left, BrickRed] (v3) to node[midway,below] {$w_{e_3}\downarrow$} (v1);
          \draw[->,bend left, BrickRed] (v1) to node[midway,above] {$w_{e_4}\downarrow$} (v3);
          \node at (0,-1.2 * \scale) {$\beta(v_1)=\frac{1}{2}$};
          \node at (2 * \scale,4.2 * \scale) {$\beta(v_2)=\frac{1}{2}$};
          \node at (3.8 * \scale,-1.2 * \scale) {$\beta(v_3)=\frac{1}{2}$};
      \end{scope}
    
      \begin{scope}[xshift=2.5cm]
        \node[inner sep=0pt] (russell) at (1,1){\includegraphics[width=0.7\textwidth]{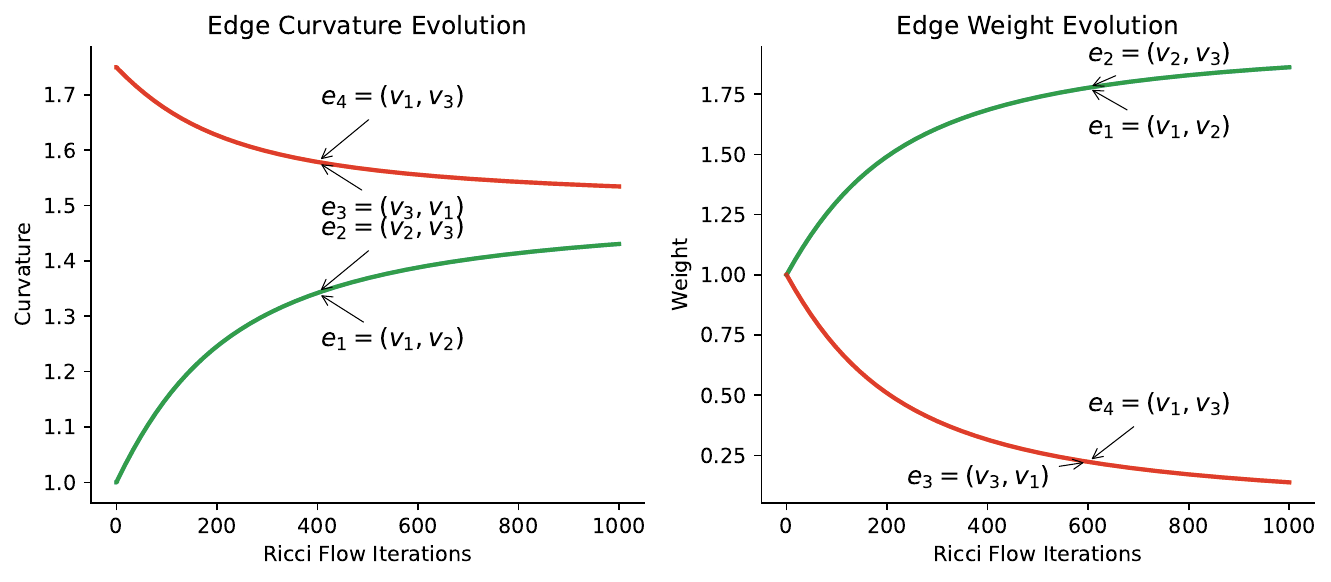}};
      \end{scope}

      \draw[-, dashed, Gray] (-6.5, -2) -- (9,-2);
      \begin{scope}[xshift=-6cm, yshift=-6cm]
          \node[vertex] (v1) at (0,0) {$v_1$};
          \node[vertex] (v2) at (2 * \scale,3.4 * \scale) {$v_2$};
          \node[vertex] (v3) at (4 * \scale,0) {$v_3$};
          \draw[->,bend left, OliveGreen] (v1) to node[midway,left] {$w_{e_1}\uparrow$} (v2);
          \draw[->,bend left, OliveGreen] (v2) to node[midway,right] {$w_{e_2}\uparrow$} (v3);
          \draw[->,bend left, BrickRed] (v3) to node[midway,below] {$w_{e_3}\downarrow$} (v1);
          \draw[->,bend left, BrickRed] (v1) to node[midway,above] {$w_{e_4}\downarrow$} (v3);
          \node at (0,-1.2 * \scale) {$\beta(v_1)=\frac{2}{3}$};
          \node at (2 * \scale,4.2 * \scale) {$\beta(v_2)=\frac{1}{2}$};
          \node at (3.8 * \scale,-1.2 * \scale) {$\beta(v_3)=\frac{1}{3}$};
      \end{scope}
    
      \begin{scope}[xshift=2.5cm, yshift=-6cm]
        \node[inner sep=0pt] (russell) at (1,1){\includegraphics[width=0.7\textwidth]{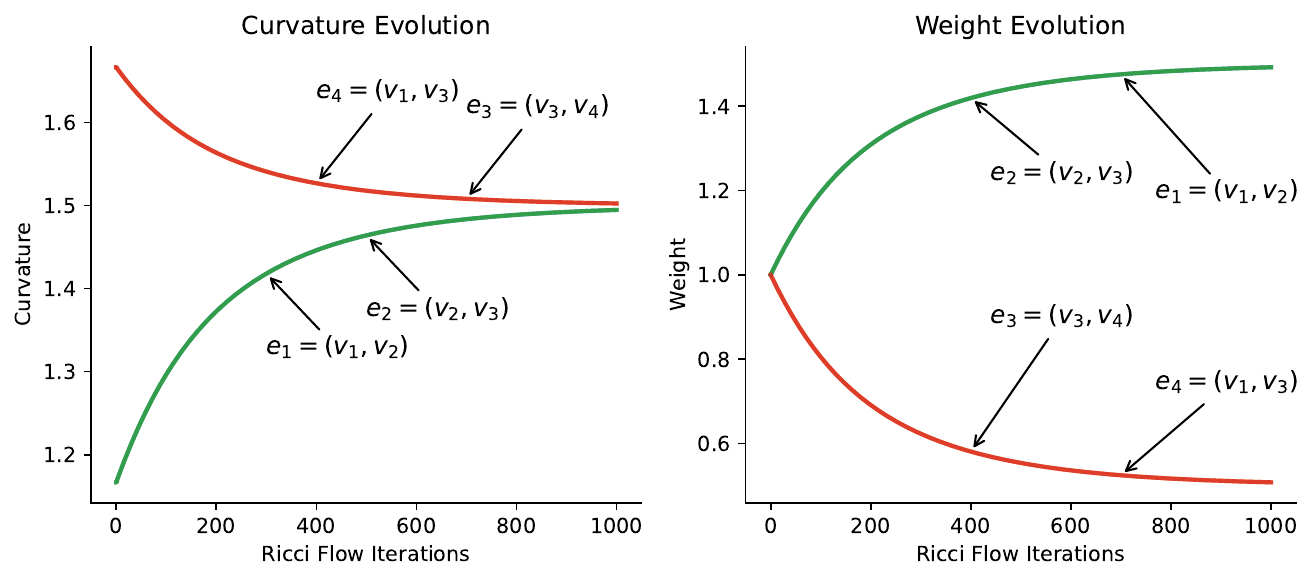}};
      \end{scope}

      \draw[-, dashed, Gray] (-6.5, -8) -- (9,-8);
      \begin{scope}[xshift=-6cm, yshift=-12cm]
          \node[vertex] (v1) at (0,0) {$v_1$};
          \node[vertex] (v2) at (2 * \scale,3.4 * \scale) {$v_2$};
          \node[vertex] (v3) at (4 * \scale,0) {$v_3$};
          \draw[->,bend left, RoyalBlue] (v1) to node[midway,left] {$w_{e_1}\uparrow$} (v2);
          \draw[->,bend left, Green] (v2) to node[midway,right] {$w_{e_2}\downarrow$} (v3);
          \draw[->,bend left, Maroon] (v3) to node[midway,below] {$w_{e_3}\downarrow$} (v1);
          \draw[->,bend left, Orange] (v1) to node[midway,above] {$w_{e_4}\downarrow$} (v3);
          \node at (0,-1.2 * \scale) {$\beta(v_1)=0.2$};
          \node at (2 * \scale,4.2 * \scale) {$\beta(v_2)=0.2$};
          \node at (3.8 * \scale,-1.2 * \scale) {$\beta(v_3)=0.2$};
      \end{scope}
    
      \begin{scope}[xshift=3cm, yshift=-12cm]
        \node[inner sep=0pt] (russell) at (1,1){\includegraphics[width=0.7\textwidth]{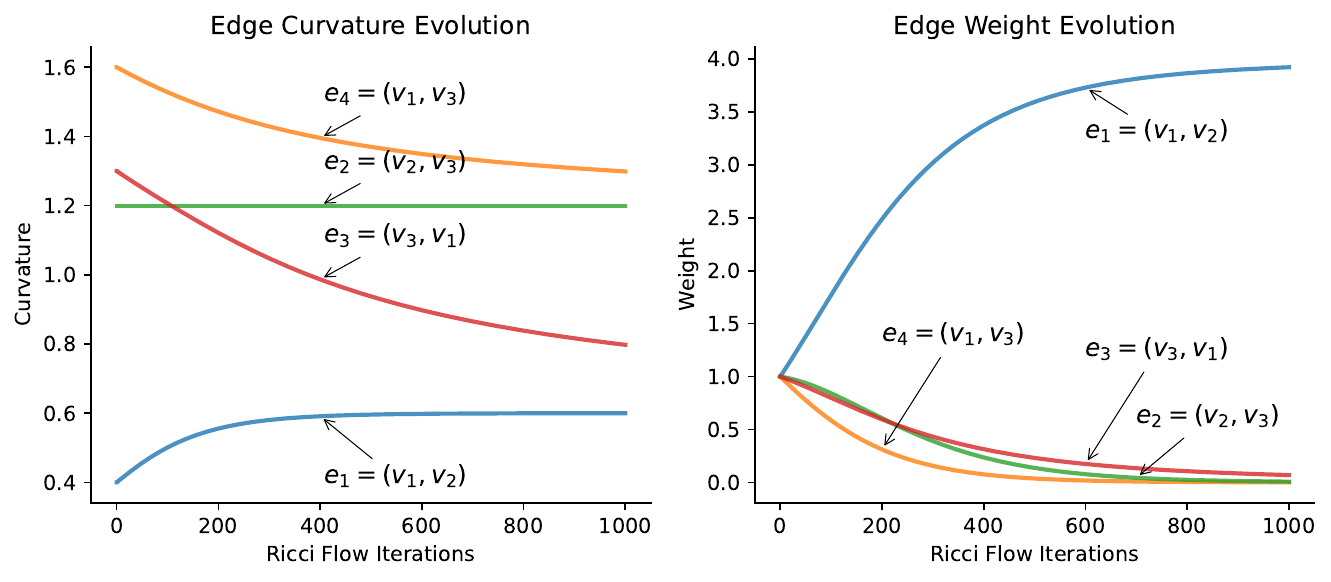}};
      \end{scope}

      \draw[-, dashed, Gray] (-6.5, -14) -- (9,-14);
      \begin{scope}[xshift=-6cm, yshift=-18cm]
          \node[vertex] (v1) at (0,0) {$v_1$};
          \node[vertex] (v2) at (2 * \scale,3.4 * \scale) {$v_2$};
          \node[vertex] (v3) at (4 * \scale,0) {$v_3$};
          \draw[->,bend left, RoyalBlue] (v1) to node[midway,left] {$w_{e_1}\downarrow$} (v2);
          \draw[->,bend left, Green] (v2) to node[midway,right] {$w_{e_2}\uparrow$} (v3);
          \draw[->,bend left, Maroon] (v3) to node[midway,below] {$w_{e_3}\downarrow$} (v1);
          \draw[->,bend left, Orange] (v1) to node[midway,above] {$w_{e_4}\downarrow$} (v3);
          \node at (0,-1.2 * \scale) {$\beta(v_1)=0.8$};
          \node at (2 * \scale,4.2 * \scale) {$\beta(v_2)=0.8$};
          \node at (3.8 * \scale,-1.2 * \scale) {$\beta(v_3)=0.8$};
      \end{scope}
    
      \begin{scope}[xshift=3cm, yshift=-18cm]
        \node[inner sep=0pt] (russell) at (1,1){\includegraphics[width=0.7\textwidth]{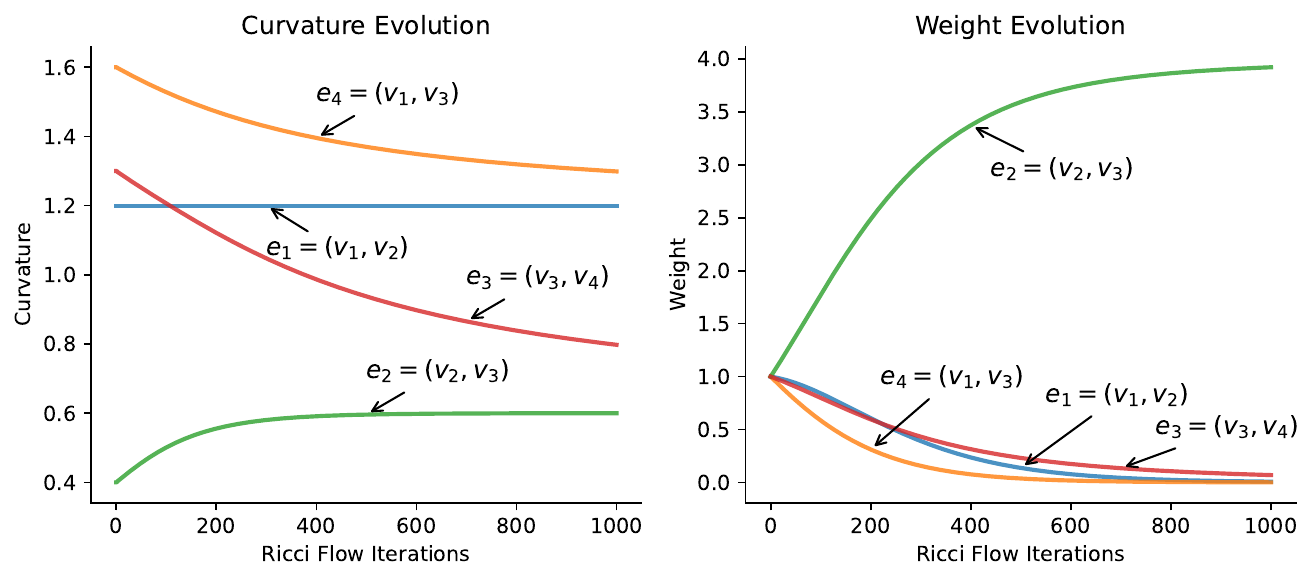}};
      \end{scope}
    \end{tikzpicture}
    \end{center}
    \caption{Asymmetric 3-cycle with extra edge. By changing the choice of $\beta(v_i), i\in\{1,2,3\}$, 
    the evolution of edge weights and curvatures are different.
    Each row corresponds to one choice of $\beta$.}
    \label{fig:3cycle4edge_evolution}
\end{figure}

\textbf{Case 1: $\bm{\beta = 0.5 }$.}
    All nodes treat out-neighborhood and in-neighborhood equally. 
    When $t=0$, the Wasserstein distances are given by
    \begin{align*}
    W_1(\mu_{v_1}^\alpha,\mu_{v_2}^\alpha) &
    = [\alpha - \tfrac{1}{2}(1-\alpha)] d(v_1, v_2) + \tfrac{1}{4} (1 - \alpha) d(v_3, v_2) 
    = \alpha,\\
    W_1(\mu_{v_2}^\alpha,\mu_{v_3}^\alpha) &
    = \tfrac{1}{4} (1 - \alpha)d(v_2, v_1) + [\alpha - \tfrac{1}{2}(1-\alpha)]d(v_2, v_3)
    = \alpha ,\\
    W_1(\mu_{v_3}^\alpha,\mu_{v_1}^\alpha) &
    = [\alpha - \tfrac{3}{4}(1-\alpha)]d(v_ 3, v_1) 
    = \alpha - \tfrac{3}{4}(1-\alpha),\\
    W_1(\mu_{v_1}^\alpha,\mu_{v_3}^\alpha) &
    = [\alpha - \tfrac{3}{4}(1-\alpha)]d(v_ 1, v_3) 
    = \alpha - \tfrac{3}{4}(1-\alpha).
    \end{align*}
    leading to
    $$\kappa_\alpha(v_1,v_2) = 1-\alpha,\quad
    \kappa_\alpha(v_2,v_3) = 1-\alpha,\quad
    \kappa_\alpha(v_3,v_1) = \tfrac{7}{4}(1-\alpha),\quad
    \kappa_\alpha(v_1, v_3) = \tfrac{7}{4}(1-\alpha).$$
    Thus, the Lin-Lu-Yau curvatures are
    $$\kappa(v_1,v_2) = \kappa(v_2, v_3) = 1,\quad
    \kappa(v_3,v_1) = \kappa(v_1,v_3) = \tfrac{7}{4},$$
    and the average curvature is $\overline{\kappa}_\sigma = \tfrac{11}{8}$.
    According to the Ricci flow equation, we have for each edge $e_i\in\gE$:
    $$w_e(t) = w_e^0 \exp\big(-t(\kappa_e - \overline{\kappa}_\sigma)\big).$$
    Specifically, as $t$ increases from $t=0$, $w_{e_1}$ and $w_{e_2}$ increase simultaneously, while $w_{e_3}$ and $w_{e_4}$ decrease simultaneously. 
    See~\figref{fig:nonsymmetric_triangle}(First row), note that the nonsymmetric essence of the graph is not reflected in this case, as the flow pattern is symmetric, 
    i.e., $w_{e_1}$ and $w_{e_2}$ behave exactly the same, and so do $w_{e_3}$ and $w_{e_4}$.

\textbf{Case 2: $\beta(x) = \frac{d^{out}_x}{d_x} $.}
  Each node treats out-neighbors and in-neighbors proportionally to their degrees, i.e., every edge, regardless of direction, is treated equally.
  Specifically, $\beta(v_1)=\tfrac{2}{3},\; \beta(v_2)=\tfrac{1}{2},\; \beta(v_3)=\tfrac{1}{3}$.
  At $t=0$, the Wasserstein distances are:
    \begin{align*}
    W_1(\mu_{v_1}^\alpha,\mu_{v_2}^\alpha) &
    = [\alpha - \tfrac{1}{2}(1-\alpha)] d(v_1, v_2) + \tfrac{1}{6} (1 - \alpha) d(v_3, v_2) 
    = \alpha - \tfrac{1}{6}(1-\alpha),\\
    W_1(\mu_{v_2}^\alpha,\mu_{v_3}^\alpha) &
    = \tfrac{1}{6} (1 - \alpha)d(v_2, v_1) + [\alpha - \tfrac{1}{2}(1-\alpha)]d(v_2, v_3)
    = \alpha - \tfrac{1}{6}(1-\alpha) ,\\
    W_1(\mu_{v_3}^\alpha,\mu_{v_1}^\alpha) &
    = [\alpha - \tfrac{2}{3}(1-\alpha)]d(v_ 3, v_1) 
    = \alpha - \tfrac{2}{3}(1-\alpha),\\
    W_1(\mu_{v_1}^\alpha,\mu_{v_3}^\alpha) &
    = [\alpha - \tfrac{2}{3}(1-\alpha)]d(v_ 1, v_3) 
    = \alpha - \tfrac{2}{3}(1-\alpha).
    \end{align*}
    leading to
    $$\kappa_\alpha(v_1,v_2) = \kappa_\alpha(v_2,v_3) = 1 - \alpha + \tfrac{1}{6}(1-\alpha),\quad
    \kappa_\alpha(v_3,v_1) = \kappa_\alpha(v_1, v_3) = 1 - \alpha + \tfrac{2}{3}(1-\alpha).$$
    Thus, the Lin-Lu-Yau curvatures are
    $$\kappa(v_1,v_2) = \kappa(v_2, v_3) = \tfrac{7}{6},\quad
    \kappa(v_3,v_1) = \kappa(v_1,v_3) = \tfrac{5}{3},$$
    and the average curvature is $\overline{\kappa}_\sigma = \tfrac{17}{3}$.

    According to the Ricci flow equation, similar to Case 1, $w_{e_1}$ and $w_{e_2}$ increase simultaneously, while $w_{e_3}$ and $w_{e_4}$ decrease simultaneously.
    See~\figref{fig:nonsymmetric_triangle}(Second row), the nonsymmetric essence of the graph is still not reflected in this case, as the flow pattern is still symmetric.
    
    Compare to Case 1, by observing the evolution of weights, we find that curvatures and weights in Case 2 converge faster than those in Case 1, i.e., they seem to reach a steady state in a shorter time.
    
\textbf{Case 3: $\beta < \tfrac{1}{2}$.}
    With $\beta$ smaller than 0.5, nodes favor inflow-based transport. 
    Setting $\beta(v_1)=\beta(v_2)=\beta(v_3)=0.2$, at $t=0$, the Wasserstein distances are:
    \begin{align*}
    W_1(\mu_{v_1}^\alpha,\mu_{v_2}^\alpha) &
    = [\alpha - \tfrac{4}{5}(1-\alpha)] d(v_1, v_2) + \tfrac{7}{10} (1 - \alpha) d(v_3, v_2) 
    = \alpha + \tfrac{3}{5}(1-\alpha),\\
    W_1(\mu_{v_2}^\alpha,\mu_{v_3}^\alpha) &
    = \tfrac{1}{5} (1 - \alpha) d(v_1, v_3) + [\alpha - \tfrac{4}{5}(1-\alpha)]d(v_2, v_3)
    = \alpha + \tfrac{1}{5}(1-\alpha) ,\\
    W_1(\mu_{v_3}^\alpha,\mu_{v_1}^\alpha) &
    = \tfrac{3}{10}(1-\alpha)d(v_2, v_1) + [\alpha - \tfrac{9}{10}(1-\alpha)]d(v_ 3, v_1)
    = \alpha - \tfrac{3}{10}(1-\alpha),\\
    W_1(\mu_{v_1}^\alpha,\mu_{v_3}^\alpha) &
    = \tfrac{3}{10}(1-\alpha)d(v_1, v_2) + [\alpha - \tfrac{9}{10}(1-\alpha)]d(v_ 1, v_3)
    = \alpha - \tfrac{3}{5}(1-\alpha).
    \end{align*}
    leading to
    $$\kappa_\alpha(v_1,v_2) = \tfrac{2}{5}(1-\alpha),\quad
    \kappa_\alpha(v_2,v_3) = \tfrac{6}{5}(1-\alpha),\quad
    \kappa_\alpha(v_3,v_1) = \tfrac{13}{10}(1-\alpha),\quad
    \kappa_\alpha(v_1, v_3) = \tfrac{8}{5}(1-\alpha).$$
    Thus, the Lin-Lu-Yau curvatures are
    $$\kappa(v_1,v_2) = \tfrac{2}{5},\quad
    \kappa(v_2, v_3) = \tfrac{6}{5},\quad
    \kappa(v_3,v_1) = \tfrac{13}{10},\quad
    \kappa(v_1,v_3) = \tfrac{8}{5},$$
    and the average curvature is $\overline{\kappa}_\sigma = \tfrac{9}{8}$. Only $\kappa(v_1,v_2)$ is smaller than $\overline{\kappa}_\sigma$.
    According to the Ricci flow equation, as $t$ increases from $t=0$, $w_{e_1}$ increases, while $w_{e_2}$, $w_{e_3}$ and $w_{e_4}$ decrease separately.
    
    See~\figref{fig:nonsymmetric_triangle}(Third row), the nonsymmetric essence of the graph is now reflected in this case, as the evolution of each edge has its own pattern.
    This indicates that the balancing factor $\beta$ can effectively model and detect the asymmetry in the graph and influence the redistribution of edge weights according to local node natures.

\textbf{Case 4: $\beta > \tfrac{1}{2}$.}
    Conversely, when $\beta$ exceeds 0.5, nodes prioritize outflow-based transport. 
    Setting $\beta(v_1)=\beta(v_2)=\beta(v_3)=0.8$, at $t=0$, the Wasserstein distances are:
    \begin{align*}
    W_1(\mu_{v_1}^\alpha,\mu_{v_2}^\alpha) &
    = [\alpha - \tfrac{2}{5}(1-\alpha)] d(v_1, v_2) + \tfrac{1}{5} (1 - \alpha) d(v_1, v_3) 
    = \alpha - \tfrac{1}{5}(1-\alpha),\\
    W_1(\mu_{v_2}^\alpha,\mu_{v_3}^\alpha) &
    = \tfrac{7}{10} (1 - \alpha) d(v_2, v_1) + [\alpha - \tfrac{4}{5}(1-\alpha)]d(v_2, v_3)
    = \alpha - \tfrac{3}{5}(1-\alpha) ,\\
    W_1(\mu_{v_3}^\alpha,\mu_{v_1}^\alpha) &
    = [\alpha - \tfrac{9}{10}(1-\alpha)]d(v_3, v_1) + \tfrac{3}{10}(1-\alpha)d(v_3, v_2)
    = \alpha - \tfrac{3}{10}(1-\alpha),\\
    W_1(\mu_{v_1}^\alpha,\mu_{v_3}^\alpha) &
    = [\alpha - \tfrac{9}{10}(1-\alpha)]d(v_1, v_3) + \tfrac{3}{10}(1-\alpha)d(v_2, v_3)
    = \alpha - \tfrac{3}{5}(1-\alpha).
    \end{align*}
    leading to
    $$\kappa_\alpha(v_1,v_2) = \tfrac{6}{5}(1-\alpha),\quad
    \kappa_\alpha(v_2,v_3) = \tfrac{2}{5}(1-\alpha),\quad
    \kappa_\alpha(v_3,v_1) = \tfrac{13}{10}(1-\alpha),\quad
    \kappa_\alpha(v_1, v_3) = \tfrac{8}{5}(1-\alpha).$$
    Thus, the Lin-Lu-Yau curvatures are
    $$\kappa(v_1,v_2) = \tfrac{6}{5},\quad
    \kappa(v_2, v_3) = \tfrac{2}{5},\quad
    \kappa(v_3,v_1) = \tfrac{13}{10},\quad
    \kappa(v_1,v_3) = \tfrac{8}{5},$$
    and the average curvature is $\overline{\kappa}_\sigma = \tfrac{9}{8}$. Only $\kappa(v_2,v_3)$ is smaller than $\overline{\kappa}_\sigma$.
    
    According to the Ricci flow equation, as $t$ increases from $t=0$, $w_{e_2}$ increases, while $w_{e_1}$, $w_{e_3}$ and $w_{e_4}$ decrease individually.
    
    As shown in~\figref{fig:nonsymmetric_triangle}(Fourth row), the nonsymmetric essence of the graph is also reflected in this case, as the evolution of each edge has its own pattern.
    
    Compare to Case 3, by observing the third and fourth rows of~\figref{fig:nonsymmetric_triangle}, 
    we find that the evolution patterns of $w_{e_2}$ and $w_{e_1}$ are opposite, 
    while the evolution patterns of $w_{e_3}$ and $w_{e_4}$ are the same,
    which aligns with the fact that $\beta$ values are complementary (i.e., $0.2 + 0.8 = 1$) and the subgraph formed by the edges $e_3$ and $e_4$ is symmetric.
\end{example}
\begin{remark}
The above example illustrates how the balancing factor $\beta$ influences the Ricci curvature and flow on a directed graph, highlighting its role in capturing the asymmetry of directed edges.
By adjusting $\beta$, we can model different transport preferences at each node, leading to distinct models of flow evolution.
We only presented four specific cases for clarity, but in practice, $\beta$ can be set to any value in $[0,1]$ for each node.
This flexibility is crucial for accurately representing and analyzing real-world directed networks, where nodes may have varying tendencies towards sending or receiving flow.
\end{remark}

Apart from the flexibility incorporated by the balancing factor $\beta$, another key feature of our proposed directed Ricci flow is that the weights evolving while the distances remain unchanged.
For application perspective, this is crucial in scenarios where the physical or geographical distances between nodes are fixed, but the weights (e.g., traffic loads, communication costs) can vary over time.
Other graph Ricci flow approaches that conflate edge weights with distances cannot capture this distinction, limiting their utility for real-world dataset modeling, while our method effectively addresses this challenge.

\section{Conclusion}\label{sec:conclusion}
We proposed a Ricci flow framework for \emph{weighted directed graphs}, proved the existence and uniqueness of its solution, introduced a node-wise balancing factor $\beta$ to capture directional asymmetry, and designed a distance-preserving flow that decouples geometry from dynamics. This approach better models real-world systems where physical distances remain fixed but flow conditions evolve.

Future research will proceed along both theoretical and practical directions. Theoretically, further investigation into the convergence behavior of the proposed Ricci flow and a deeper analysis of the role of the balancing factor $\beta$ in shaping network evolution are of particular interest. Practically, applying this framework to real-world networks, such as transportation systems, communication networks, and microbial interaction networks, etc., will uncover more geometric structural properties.

\section*{Acknowledgments}
Shuliang Bai is supported in part by NSFC grant number 12301434. 
Shuang Liu is supported by the National Nature Science Foundation of China (Grant No.12001536, 12371102).
Xin Lai is supported by the start-up research fund from Beijing Institute of Mathematical Sciences and Applications (BIMSA).

\printbibliography

\end{document}